\pdfoutput=1

\documentclass[10pt]{amsart}

\usepackage[english]{babel}

\usepackage{amsmath}

\usepackage{amssymb}

\usepackage{graphicx}

\newtheorem{thm}{Theorem}[section]

\newtheorem{thmx}{Theorem} 

\newtheorem{prop}[thm]{Proposition}
\newtheorem{lemma}[thm]{Lemma}

\theoremstyle{definition}
\newtheorem{defn}{Definition}[section]
\newtheorem{ex}{Example}

\theoremstyle{remark}
\newtheorem{remark}{Remark}

\numberwithin{equation}{section}

\def\C{\mathbb{C}}
\def\R{\mathbb{R}}

\newcommand{\RedeclareMathOperator}[2]{%
        \let#1\undefined
        \DeclareMathOperator{#1}{#2}}

\DeclareMathOperator{\cn}{cn}
\DeclareMathOperator{\dn}{dn}

\DeclareMathOperator*{\contfrac}{\lower 0.4em\hbox{\larger[4]\bf K}}

\renewcommand{\Re}{{\rm Re}\,}          
\renewcommand{\Im}{{\rm Im}\,}          

\begin{document}

\title[Quantization of the conformal arclength functional]{Quantization of the conformal arclength
functional on space curves}

\author{Emilio Musso}
\address{(E. Musso) Dipartimento di Scienze Matematiche, Politecnico di Torino,
Corso Duca degli Abruzzi 24, I-10129 Torino, Italy}
\email{emilio.musso@polito.it}

\author{Lorenzo Nicolodi}
\address{(L. Nicolodi) Di\-par\-ti\-men\-to di Scienze Ma\-te\-ma\-ti\-che, Fisiche e Informatiche,
Uni\-ver\-si\-t\`a di Parma, Parco Area delle Scienze 53/A, I-43124 Parma, Italy}
\email{lorenzo.nicolodi@unipr.it}

\thanks{Authors partially supported by
PRIN 2010-2011 ``Variet\`a reali e complesse: geometria, to\-po\-lo\-gia e analisi ar\-mo\-ni\-ca'';
FIRB 2008 ``Geometria Differenziale Complessa e Dinamica Olomorfa'';
and the GNSAGA of INDAM}




\begin{abstract}
By a \textit{conformal string} in Euclidean space
is meant a closed critical curve with non-constant conformal curvatures
of the conformal arclength functional.
We prove that (1) the set of conformal classes of conformal strings
is in 1-1 correspondence with the rational points of the complex domain
$\{q\in \mathbb{C} \,:\, {1}/{2} < \Re q< {1}/{\sqrt{2}},\,\, \Im q>0,\,\, |q| < {1}/{\sqrt{2}}\}$
and (2) any conformal class has a model conformal string, called \textit{symmetrical configuration},
which is determined by three phenomenological invariants: the order of its symmetry group and
its linking numbers with the two conformal circles representing the rotational axes of
the symmetry group.
This amounts to the quantization of closed trajectories of
the contact dynamical system associated to the conformal arclength functional
via Griffiths' formalism of the calculus of variations.
\end{abstract}

\maketitle


\section*{Introduction}\label{s:intro}

The M\"obius geometry of space curves was mainly developed in the first half of the past
century \cite{F,L,T,V} and later taken up starting from the early 1980's \cite{CSW,LO2010,SS,S}.
Further developments of the subject as an instance of the conformal geometry of submanifolds
can be found in \cite{BC} and the literature therein.
The subject has also received much attention for its many fields of application, including
the theory of integrable systems \cite{CQ,EB}, the topology and M\"obius energy of
knots \cite{BFHW,FHW,LO}, and the geometric approach to shape analysis and medical
imaging \cite{OS-Med-im}.

Let $\gamma\subset\R^n$, $n\geq 3$, be a smooth curve parametrized by arclength $s$. The
{\em conformal arclength} parameter $\zeta$ of $\gamma$ is defined (up to a constant) by
\begin{equation}\label{inf-conf-par}
 d\zeta =  \left(\langle\dddot\gamma,\dddot\gamma\rangle
    -\langle\ddot\gamma, \ddot\gamma\rangle^2 \right)^{\frac{1}{4}}ds =: \eta_\gamma,
 \end{equation}
%
%
where $\langle\,,\rangle$ is the standard scalar product on $\R^n$.
The 1-form $\eta_\gamma$, the {\em infinitesimal conformal arclength} of $\gamma$,
is conformally invariant.
If ${\eta_\gamma}{\vert_s} \neq 0$, for each $s$, the curve
is called {\em generic}.\footnote{Generic curves, either closed or not, constitute an
open dense subset of all smooth curves in the $C^\infty$ topology (cf. \cite{CSW}).}
The conformal arclength $\zeta$ gives a conformally invariant parametrization of a generic curve.
We consider the conformally invariant variational problem on generic
curves defined by the {\em conformal arclength functional}
\begin{equation}\label{var-pbm}
 \mathcal{L}[\gamma] = \int_\gamma \eta_\gamma.
  \end{equation}
This variational problem was studied in \cite{M1} for $n=3$
and more recently in \cite{MMR} for higher dimensions.\footnote{We
adhere to the standard terminology adopted
for $\mathcal L$ in the literature.
However, observe
that
$\left(\langle\dddot\gamma,\dddot\gamma\rangle
    -\langle\ddot\gamma, \ddot\gamma\rangle^2 \right)^{1/4}$ has the dimension $L^{-1}$,
    so that $\eta_\gamma$ is dimensionless.}
Accordingly,
the critical curves can be found by quadratures and explicit parametrizations are given
in terms of elliptic functions and integrals.
In this paper we address the question of existence and properties
of closed critical curves for the functional $\mathcal L$.
Actually, it suffices to consider the 3-dimensional case only,
since from the results in \cite{MMR}
we can see that
any closed critical curve in $\R^n$ lies in some
$\R^3 \subset \R^n$, up to a conformal transformation.
It is known that a generic space curve is determined,
up to conformal transformations, by the conformal arclength and two
conformal curvatures (cf. Section \ref{s:1}).
As for a closed critical
curve with constant conformal curvatures, one can see that it is conformally equivalent
to a closed rhumb line (loxodrome) of a torus of revolution (cf. Example \ref{e:1}).

\vskip0.1cm

The purpose of this paper is to study the class of closed critical curves with non-constant
conformal curvatures,
for brevity called {\it conformal strings}.
We begin by describing our three main results.
If $\gamma$, $\widetilde{\gamma}:\R \to \R^3$ are two curves
and
$[\gamma]$, $[\widetilde{\gamma}]$ denote their trajectories,
then $\gamma$ and $\widetilde{\gamma}$ are said {\it M\"obius} ({\it conformally}) {\it equivalent}
if there is an element $A$ of the M\"obius group $G$ of $\mathbb R^3$, such
that $A\cdot [\gamma]=[\widetilde{\gamma}]$.
By a {\it conformal symmetry} of a curve
$\gamma$ is meant
an element $A\in G$, such that $A\cdot [\gamma]=[\gamma]$. The set of all symmetries of $\gamma$ is
a subgroup $G_{\gamma}$ of $G$.
The symmetry group of a closed curve
with non-constant conformal curvatures is finite and its cardinality is called the
{\it symmetry index} of $\gamma$.
Our first main result is the following.

\begin{thmx}\label{thm:A}
The M\"obius classes
of conformal strings are in 1-1 correspondence
with the rational points of the
complex domain
\[
 \Omega=
 \left\{ q\in \mathbb{C} \,\,:\,\, \frac{1}{2} < \Re q<\frac{1}{\sqrt{2}},\,\, \Im q>0,\,\,
  |q|< \frac{1}{\sqrt{2}} \right\}.
  \]
The rational points of $\Omega$ are called the {\em moduli} of conformal strings.

\end{thmx}

Using this theorem and other technical results, we will prove that
any M\"obius class of strings is represented by a model string. This is
our second main result.

\begin{thmx}\label{thm:B}
The conformal strings corresponding to a modulus $q\in \Omega$ are M\"obius
equivalent to
a model string $\gamma_{q}=(x(t),y(t),z(t)) : \R\to \R^3$,
\begin{equation}
\begin{cases}
x(t)=\frac{\sqrt{2}}{r(t)}\mu\sqrt{k(t)^2-\upsilon^2}\cos \Theta_2(t) ,\\
y(t)=\frac{\sqrt{2}}{r(t)}\mu\sqrt{k(t)^2-\upsilon^2}\sin \Theta_2(t) ,\\
z(t)=\frac{\sqrt{2}}{r(t)}\upsilon\sqrt{\mu^2-k(t)^2}\sin \Theta_1(t),\\
\end{cases}
\end{equation}
called the {\em symmetrical configuration} of $q$.
 Here
\[
\begin{aligned}
k(t) &=
\left\{\!\!\!
\begin{array}{lr}
\sqrt{a} \cn\big(\sqrt{a-b}\,t, \frac{a}{a-b}\big),& b<0,\\
 \sqrt{a}\dn\big(\sqrt{a}\,t,\frac{a-b}{a}\big), &  b>0,
\end{array}
\right.\\
r(t)&=\sqrt{\mu^2-\upsilon^2}k(t)+\upsilon\sqrt{\mu^2-k(t)^2} \cos \Theta_1(t),\\
\end{aligned}
\]
where $\mu=\frac{1}{\sqrt{2}}\sqrt{a+b+\sqrt{4+(a-b)^2}}$,
$\upsilon=\frac{1}{\sqrt{2}}\sqrt{a+b-\sqrt{4+(a-b)^2}}$,
\[
 \Theta_1(t)=\int_0^t\frac{\mu}{\mu^2-k(u)^2}du,\quad
   \Theta_2(t)=\int_0^t\frac{\upsilon}{\upsilon^2-k(u)^2}du,
   \]
   and
where $a$ and $b$
are real parameters, uniquely defined by $q$, such that $a>0$, $a>b$, $b\neq0$, and $ab>1$.

\end{thmx}

The symmetry group of a symmetrical configuration has a special structure,
which is described by the following.

\begin{thmx}\label{thm:C}
Let $\gamma_{q} : \R\to \R^3$ be the symmetrical configuration
corresponding to the modulus $q = q_1 +i q_2$, where $q_1=m_1/n_1$, $q_2=m_2/n_2$,
and the pairs $(m_1,n_1)$, $(m_2,n_2)$ are coprime integers.
Let $n$ be the least common multiple of $n_1$ and $n_2$, and consider the coprime
integers $h_1=n/n_1$ and $h_2=n/n_2$. Then,
\begin{enumerate}
\item $n$ is the order of the symmetry group of $\gamma_{q}$;

\item
$m_1h_1$ and $m_2h_2$
are the linking numbers of $\gamma_{q}$ with the Clifford circle
\[
  \mathcal{C}=\left\{(x,y,0) \in\R^3 : x^2+y^2=2\right\}
    \]
and the $z$-axis, respectively.
\end{enumerate}
\end{thmx}

An important consequence of the previous results is that the conformal shape of a string
is uniquely determined by three phenomenological invariants: the order of its symmetry
group and its linking numbers with the two axes of the symmetry group.
The explicit construction of a string from the phenomenological
invariants requires the inversion of the period map of $\mathcal L$ (cf. Section \ref{s:2}).
In this respect, some numerical experiments carried out with the software {\sl Mathematica}
suggest that conformal strings are simple curves (cf. Section \ref{s:6}).
However, a rigorous proof of this fact is still missing.
Another interesting problem is to find an estimate for the asymptotic growth of $\varrho(n)$,
the cardinality  of the set of M\"obius classes of strings with symmetry order $n$.
Numerical experiments suggest a quadratic growth of $\varrho(n)$.

\vskip0.1cm

The theorems above have a conceptual explanation within the general scheme of
Griffiths' formalism of the calculus of variations \cite{GM,Gr,MNforum,MN-CCG,MN-SICON,MN-JMIV,MN}.
Using Griffiths' formalism, the {\it momentum space} of the variational problem can
be identified with $Y=G_+\times \mathcal{A}$, where $G_+$ is the identity
component of the M\"obius group of $\R^3$ and $\mathcal{A}$ is a 3-dimensional
submanifold of $\mathfrak{g}^*$, the dual of the Lie algebra of $G_+$. Moreover,
the restriction $\xi\in \Omega^1(Y)$ of the Liouville form of $T^*(G_+)$ defines
an invariant contact structure on $Y$. By choosing a suitable set
of coordinates on $\mathcal{A}$, say $p_1$, $p_2$ and $p_3$, the
characteristic curves of the contact form are given by
\[
  t\in \R\mapsto (\mathrm{F}(t),p_1(t),p_2(t),p_3(t))\in Y,
   \]
where $\mathrm{F}$ is the canonical lift of a stationary curve $\gamma$,
$p_1$ and $p_2$ are the conformal curvatures, and $p_3=p_1'$. From a
theoretical point of view, the study of the variational problem is equivalent to
that of the dynamical system defined by the characteristic vector field of the
contact form $\xi$. One can easily see that the contact momentum
map
is given by
\[
 J:(\mathrm{F},p)\in Y\mapsto \mathrm{Ad}^{*}(F)(\xi|_{(\mathrm{F},p)})\in \mathfrak{g}^*.
 \]
(For the construction of the momentum
map in contact geometry, see for instance \cite{OR}.)
Moreover, the action of $G_+$ on $Y$ is Hamiltonian and coisotropic, and the characteristic
vector field is {\it collective completely integrable} \cite{FT,GS2,J}.\footnote{Here,
we adopt the
terminology used in \cite{GS2}. In the literature, the term
non-commutative completely integrable systems is also used.} Consequently,
the flow can be linearized on the fibers of the momentum map and its trajectories
can be found by quadratures (see \cite{GM} for a general description of the
integration procedure).
Theorems \ref{thm:A} and \ref{thm:C} say
that the contact dynamical system is quantizable, at least in the sense of
``the old quantum theory'' (Bohr's atom theory) \cite{Messiah}, and that the quantum numbers
of the closed (quantizable) trajectories have a precise geometric meaning.

\vskip0.1cm

The paper is organized as follows.
Section \ref{s:1} collects some basic facts about the conformal geometry of space curves.
Section \ref{s:2} recalls the Euler--Lagrange equations of the variational problem and
discusses the example of closed critical curves
with constant conformal curvatures.
Then, after introducing the natural parameters of a critical
curve with non-constant periodic conformal curvatures,
the period map of the functional $\mathcal L$ is defined, and
conformal strings are characterized in terms of the natural parameters,
via the period map.
Section \ref{s:3} deals with a technical result
about the period map, from which Theorem \ref{thm:A} follows directly.
Section \ref{s:4} proves Theorem \ref{thm:B}, while Section \ref{s:5}
proves Theorem \ref{thm:C}. Section \ref{s:6} discusses some examples.

\vskip0.1cm
Numerical and symbolic computations, as well as graphics, are made with the software
{\sl Mathematica}.
As basic references for the theory of elliptic functions and integrals we use the
monographs \cite{By,La} (see also \cite{W}).
For the few notions of knot theory used in the paper we refer to \cite{Ma}.
A general reference for M\"obius geometry is \cite{H},
to which we refer for an updated list of modern and classical references to the subject
(see also \cite{JMNbook}).
The main results of the paper were previously announced in
\cite{M2}.

\vskip0.1cm
(Added in proof) Since acceptance of the manuscript, the paper \cite{DMNna} has appeared, which
studies the local and global conformal invariants of timelike curves in the (1+2)-Einstein universe
and addresses the question of existence and properties of closed trajectories for the conformal strain
functional.

\subsubsection*{Acknowledgments} The authors would like to thank the
referees for their valuable comments and suggestions.

\section{Preliminaries}\label{s:1}

\subsection{The conformal group}\label{ss:1.1}

Let $\R^{4,1}$ denote $\R^5$ with the Lorentz scalar product
\begin{equation}\label{1.1.1}
 (v,w)=-(v^0w^4+v^4w^0)+\sum_{j=1}^{3} v^jw^j = \sum_{a,b=0}^{4}g_{ab}v^aw^b,\quad g_{ab}=g_{ba},
  \end{equation}
where $v=(v^0,\dots,v^4)$, and with the space and time orientations defined, respectively,
by the volume form
$dv^0\wedge\cdots\wedge dv^4$ and the positive light cone
\begin{equation}\label{1.1.2}
  \mathcal{L}_+=\{v \in \R^{4,1} \,: \,(v,v)=0,\, v^0+v^4>0\}.
   \end{equation}
The {\it M\"obius space} $\mathcal{M}_3$ is the projectivization of $\mathcal{L}_+$,
endowed with the oriented conformal structure induced by the scalar product and the space
and time orientations. If $(e_0,\dots,e_4)$ is the standard basis of $\R^{4,1}$, the map
\begin{equation}\label{1.1.3}
  \mathcal{J}:x=(x^1,x^2,x^3)\in \R^3\mapsto \left[\frac{\langle x, x\rangle}{2}e_0
   +\sum_{j=1}^{3}x^je_j+e_4\right]\in \mathcal{M}_3.
    \end{equation}
is an orientation-preserving conformal diffeomorphism of $\R^3$ onto the M\"obius space
minus the point $P_{\infty}=[e_0]$. The inverse of $\mathcal{J}$ is the conformal projection
\begin{equation}\label{1.1.4}
 \mathcal{P}:\left[\sum_{a=0}^{4}v^ae_a\right]\in
   \mathcal{M}_3\setminus\{P_{\infty}\}\mapsto \frac{1}{v^4}(v^1,v^2,v^3)\in \R^3.
     \end{equation}
The {\it M\"obius group} $G$ consists of all pseudo-orthogonal transformations
preserving the volume form. It is a 10-dimensional Lie group with two connected
components. The first component is the subgroup $G_+$ consisting of all $\mathrm{F}\in G$
preserving the positive light cone and the second one consists of all $\mathrm{F}\in G$
switching the positive light cone with the negative one. The group $G$ acts effectively
and transitively on the left of $\mathcal{M}_3$ preserving the conformal structure.
The classical Liouville theorem \cite{DNF} asserts that every conformal automorphism
of $\mathcal{M}_3$ is induced by a unique element of $G$. Consequently, the M\"obius
group can be viewed as the pseudo-group of all conformal transformations of
Euclidean 3-space. The orientation-preserving conformal transformations are induced
by the elements of $G_+$, while the conformal transformations induced the elements
of $G_-$ are orientation-reversing. For each $\mathrm{F}\in G_+$,
 we denote by $\mathrm{F}_0,\dots,\mathrm{F}_4$ its column vectors. Then,
 $(\mathrm{F}_0,\dots,\mathrm{F}_4)$
is a {\it positive light cone} basis of $\R^{4,1}$, that is a positive-oriented basis such that
\[
  (\mathrm{F}_a,\mathrm{F}_b)=g_{ab},\quad  \mathrm{F}_0,\mathrm{F}_4\in \mathcal{L}_+ ,
     \quad a,b=0,\dots,4.
    \]
Conversely, if $(\mathrm{F}_0,\dots,\mathrm{F}_4)$ is a positive light-cone basis,
then the matrix $\mathrm{F}$ with column vectors $\mathrm{F}_0,\dots,\mathrm{F}_4$ i
s an element of $G_+$. The Lie algebra of $G$ consists of all skew-adjoint matrices of the scalar product \eqref{1.1.1}, that is
\[
  \mathfrak{g}=\left\{\mathrm{X}\in \mathfrak{gl}(5,\R) \,:\, ^t\mathrm{X}  g + g  \mathrm{X}=0
  \right\},
  \]
where $g=(g_{ab})$. The maximal compact abelian subgroups of $G$ are conjugate to the 2-dimensional torus
\begin{equation}\label{1.1.5}
T =\left\{R(\phi_1,\phi_2)\,:\, \phi_1,\phi_2\in [0,2\pi)\right\}\cong SO(2)\times SO(2),
  \end{equation}
where
\begin{equation}\label{1.1.6}
   R(\phi_1,\phi_2)=\left(
   \begin{smallmatrix}
                   \frac{1+\cos \phi_2}{2} & 0 & 0 & -\frac{\sin\phi_2}{\sqrt{2}} & \frac{1-\cos\phi_2}{2} \\
                   0 & \cos\phi_1 & -\sin\phi_1 & 0 & 0 \\
                   0 & \sin\phi_1 & \cos\phi_1 & 0 & 0 \\
                   \frac{\sin\phi_2}{\sqrt{2}} & 0 & 0 & \cos\phi_2 & -\frac{\sin\phi_2}{\sqrt{2}} \\
                   \frac{1-\cos\phi_2}{2} & 0 & 0 & \frac{\sin\phi_2}{\sqrt{2}} & \frac{1+\cos\phi_2}{2} \\
                 \end{smallmatrix}
               \right).
  \end{equation}
Note that $R(\phi_1,\phi_2)$ is the composition of the Euclidean rotation
of angle $\phi_1$ around the $z$-axis with the {\it toroidal rotation}
of angle $\phi_2$ around the {\it Clifford circle}
$\mathcal{C}=\{(x,y,0) : x^2+y^2=2\}$. The $z$-axis and the Clifford circle are
the {\it rotational axes} of $T$. The rotational axes of any other
maximal torus $\mathrm{F} T  \mathrm{F}^{-1}$
 are the images under
$\mathrm{F}$ of the axes of $T$.

\subsection{M\"obius geometry of space curves}\label{ss:1.2}

Let $\gamma : I\subset \R \to \R^3$ be a smooth curve parametrized
by arclength $s$, $I$ an open interval.
Points where the infinitesimal conformal arclength $\eta_\gamma$ (cf.~\eqref{inf-conf-par})
vanishes are called {\it vertices} of $\gamma$ (cf. \cite{CSW,LO2010}).
Generic curves, i.e., without vertices, can be parametrized by the {\it conformal arc\-length}
parameter $\zeta$, defined (up to a constant) by $d\zeta=\eta_\gamma$.
If such a conformal parametrization is defined
for every $\zeta\in \R$, the curve is said {\it complete}.
A {\it frame field} along $\gamma$ is a smooth map $\mathrm{F} :I\to G_+$, such
that $\mathcal{P}\circ \mathrm{F}_4=\gamma$.
We have the following.

\begin{prop}[cf. \cite{CSW,F,MMR,M1,S}]\label{P1.2.1}
For any oriented generic curve $\gamma : I\to \R^3$, there is a unique
frame field $\mathrm{F}:I\to G_+$ along $\gamma$, the {\em Vessiot frame}, such that
\begin{equation}\label{1.2.3}
  \mathrm{F}^{-1}d\mathrm{F}=
     \left(
     \begin{smallmatrix}
                          0 & 1 & 0 & 0 & 0 \\
                          k_2 & 0 & 0 & 0 & 1 \\
                          1 & 0 & 0 & k_1 & 0 \\
                          0 & 0 & -k_1 & 0 & 0 \\
                          0 & k_2 & 1 & 0 & 0 \\
      \end{smallmatrix}
       \right)\eta_\gamma,
    \end{equation}
where $k_1$, $k_2$ are smooth functions, called the {\em conformal curvatures}.
We call $\Gamma = \mathrm{F}_4:I\to \mathcal{L}_+$ the {\em canonical null lift} of $\gamma$.
\end{prop}

\begin{remark}
If $\gamma$ is biregular and $T = \dot\gamma$, $N$, $B$ is its Frenet frame,
with Frenet--Serret equations
$\dot T = \kappa N$, $\dot N = -\kappa T + \tau B$, $\dot B = -\tau N$,
we have ${\dddot\gamma}= -\kappa^2 T +\dot \kappa N  + \kappa\tau B$ and
then $\eta_\gamma = \sqrt[4]{\dot{\kappa}^2+\kappa^2\tau^2}\, ds$ (cf. \cite{S,T}).
The conformal curvatures take the form
\begin{equation}\label{1.2.2}
  k_1=r^5\left(\kappa^2\tau^3+\kappa \dot{\kappa}\dot{\tau}
    +\tau(2\dot{\kappa}^2-\kappa \ddot{\kappa}) \right),\quad
   k_2 =\frac{1}{2}\left(\dot{r}^2 -2r\ddot{r}-r^2\kappa^2 \right),
    \end{equation}
where $r = (\dot{\kappa}^2+\kappa^2\tau^2)^{-1/4}$ (cf. \cite{S}).
For a geometric description of vertices and the conformal arclength via the
osculating circles of $\gamma$ we refer to \cite{LO2010,Mon}.
\end{remark}

\begin{remark}
The construction of $\mathrm F$ is explicit and only involves derivatives
and simplification of $\mathrm{F}^{-1}d\mathrm{F}$ by linear relations on entries.
If $\gamma : \R\to \R^3$ is a
periodic parametrization of a closed curve, $\mathrm F$ is a periodic $G_+$-valued map.
The value of $\mathrm F$ at $t\in I$ depends on the fourth order
jet of $\gamma$ at $t$.
If $\eta_\gamma = d\zeta$, any solution of the linear
system \eqref{1.2.3} is the Vessiot frame of the conformal parametrization
$\gamma=\mathcal{P}\circ [\mathrm{F}_4]$ of a generic curve with curvatures $k_1$
and $k_2$. If $\mathrm{F}$, $\widetilde{\mathrm{F}}:\R\to G_+$
are two solutions of \eqref{1.2.3}, with initial conditions in
$G_+$, there is a unique $A\in G_+$, such that $\widetilde{\mathrm{F}}=A \mathrm{F}$.
This shows that the conformal curvatures determine the curve, up to an
orientation-preserving conformal transformation.

A curve $\gamma :\R\to \R^3$
 is called {\it chiral} if its symmetry group $G_\gamma$ is contained in $G_+$.
\end{remark}

\begin{defn}\label{def:monodromy-chirality}
A generic curve $\gamma: \R \to \R^3$ parametrized by conformal arclength
is said {\it quasi-periodic} if its conformal curvatures
are non constant, periodic,
with a common minimal period $\omega >0$;
$\omega$ is called the {\it conformal wavelength} of $\gamma$.

The {\it monodromy} of a quasi-periodic curve is the element
$M := \mathrm{F}(\omega)\mathrm{F}(0)^{-1}\in G_+$. By construction, $M$ satisfies
\begin{equation}
  \mathrm{F}(t+p \omega)= M^p  \mathrm{F}(t),
    \quad \forall t\in \R,\quad \forall p\in \mathbb{Z}.
      \end{equation}
Therefore, $M$ generates a subgroup $\widehat{G}_{\gamma}\subset G_{\gamma}$, the
{\it monodromy group} of $\gamma$. If $\gamma$ is closed,
the integral of $\eta_\gamma$ along $\gamma$ is $n\omega$, where $n$ is the
cardinality of $\widehat{G}_{\gamma}$.

If $A$ is an orientation-reversing conformal transformation, the conformal curvatures
of $A\cdot \gamma$ are $-k_1$ and $k_2$, respectively. Thus, generic curves whose first
curvature is nowhere vanishing are chiral.
In this case, we assume that $\gamma$ has
{\it positive chirality} (i.e., $k_1>0$). If $\gamma$ is a real-analytic curve with
positive chirality, then $\widehat{G}_{\gamma}=G_{\gamma}$ and $\gamma$ is closed if
and only if $M$ has finite order. Note that the symmetry index of a real-analytic
closed curve with positive chirality coincides with the order of the monodromy.
\end{defn}

\section{Critical curves}\label{s:2}

\subsection{The Euler-Lagrange equations}\label{2.1}
The critical curves of the functional
$\mathcal{L}$ \eqref{var-pbm} are characterized by the Euler--Lagrange equations \cite{M1}
\begin{equation}\label{2.1.1}
  k_1''=2k_1(C_1-k_1^2),\quad k_2=-\frac{3}{2}k_1^2+C_1,
   \end{equation}
where $C_1$ is a constant of integration and $k_1''$ is the second order derivative of $k_1$
with respect to the conformal arclength $\zeta$.

\begin{ex}[Critical curves with constant conformal curvatures]\label{e:1}

Let $\gamma$ be a critical curve with constant conformal curvatures. Then,
either $k_1=0$ and $k_2\in \R$, or $k_2=-k_1^2/2$ and $k_1\in \R\setminus\{0\}$.
In the second case, we may assume $k_1>0$. The class of curves with constant
conformal curvatures was studied in \cite{S}: they are equivalent
to the rhumb lines of either a torus of revolution, or a round cone,
or else a circular cylinder. Since we deal with curves without
vertices, the meridians and the parallels must be excluded from the discussion.
The rhumb lines of a round cone, which possibly can degenerate into the
punctured plane, are helices over logarithmic spirals, while the rhumb lines of
a circular cylinder are circular helices. All of them are not closed. From the
viewpoint of the conformal geometry, any rotationally invariant
torus\footnote{And, more generally, every compact Dupin cyclide.} is equivalent
to a torus $\mathcal{T}_r$ generated by rotating around the $z$-axis the
circle in the $xz$-plane with radius $(2-r^2)/2r$ and center
$((r^2+1)/2r,0,0)$, for some $r\in (0,\sqrt{2})$. The latter are the regular
orbits of the action of the group $T$ (cf. \eqref{1.1.5}) on the Euclidean space.
If
\[
\begin{cases}
 x_r(\theta,\phi)=\frac{4r\cos\theta}{2+r^2+(2-r^2)\cos\phi},\\
  y_r(\theta,\phi)=\frac{4r\sin\theta}{2+r^2+(2-r^2)\cos\phi},\\
  z_r(\theta,\phi)=\frac{\sqrt{2}(r^2-2)\sin\phi}{2+r^2+(2-r^2)\cos\phi}
   \end{cases}
\]
are parametric equations of $\mathcal{T}_r$, its rhumb lines are
\[
 \alpha_{r,m,n}: t\in \R \mapsto (x_r(nt,mt),y_r(nt,mt),z_r(nt,mt))\in \R^3,
  \]
where $m,n\in \R$ and $mn\neq 0,n\neq \pm m$. The conformal curvatures of $\alpha_{r,m,n}$ are
\[
 \begin{cases}
  k_1=-\frac{(2+r^2)mn}{\sqrt[4]{8m^2n^2(m^2-n^2)^2r^2(2-r^2)^2}},\\
  k_2=\frac{(8n^4r^2+m^4(r^2-2)^2)\mid mn(m^2-n^2)\mid}{4\sqrt{2}rm^2n^2(m^2-n^2)(r^2-2)}.
  \end{cases}
   \]
They verify the inequality $2k_1^2 k_2<-1$. Conversely, every curve with constant
conformal curvatures that satisfies the above inequality is equivalent to a rhumb
line of $\mathcal{T}_r$, for some $r\in (0,\sqrt{2})$. Possibly, $r$ can be computed
in terms of the constants $k_1$ and $k_2$. From this it follows that, up to conformal
transformations, the only closed critical curves with constant conformal curvatures
are of the form $\alpha_{r(q),q,1}$, where $q$ is a positive rational number different
from $1$ and
\[
  r(q)=\frac{\sqrt{2}}{q}\sqrt{2+q^2-2\sqrt{1+q^2}}.
  \]
The trajectory of  $\alpha_{r(q),q,1}$ is a torus knot of type $(m,n)$, where $m,n$ are
coprime integers such that $m/n=q$, see Figure \ref{FIG1}.
\end{ex}

\begin{figure}[ht]
\begin{center}
\includegraphics[height=6cm,width=6cm]{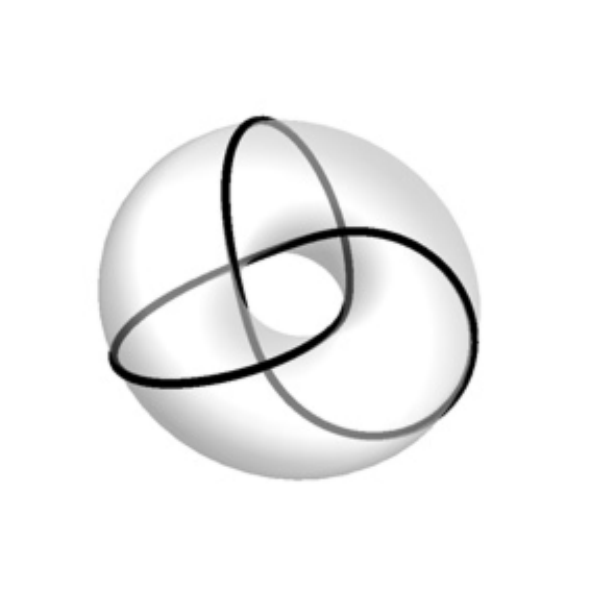}
\includegraphics[height=6cm,width=6cm]{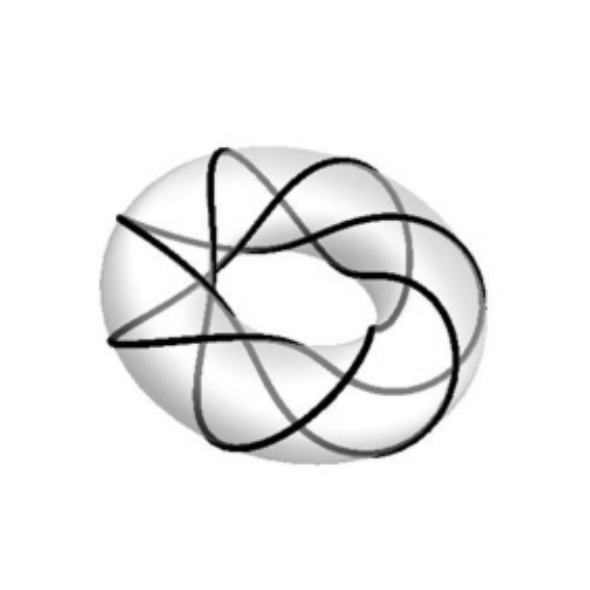}
\caption{Closed critical curves with constant conformal curvatures and torus
knots of type $(3,2)$ and $(7,3)$, respectively.}\label{FIG1}
\end{center}
\end{figure}

 From now on we consider critical curves with non-constant conformal curvatures,
parametrized by the conformal arclength parameter. Then, \eqref{2.1.1} implies
\begin{equation}\label{2.1.2}
 (k_1')^2+k_1^4-2C_1k_1^2=C_2,
   \end{equation}
where $C_2$ is another constant of integration. This equation has non-constant
periodic solutions if and only if the polynomial
$
  P(t,C_1,C_2)=t^2-2C_1t-C_2
   $
has two distinct real roots, denoted by $a,b$ and ordered in such a way that
$a>b$, with $a>0$ and $b\neq 0$. Then \eqref{2.1.1} can be rewritten in the form
\begin{equation}\label{2.1.3}
 (k_1')^2+(k_1^2-a)(k_1^2-b)=0,\quad k_2 = -\frac{3}{2}k_1^2+\frac{a+b}{2}.
  \end{equation}
The general solution of \eqref{2.1.3} is given by
\begin{equation}\label{2.1.4}
 \begin{aligned}
  k_1(t)=\epsilon_1 k_{(a,b)}(\epsilon_2 t+h),\quad k_2(t)=-\frac{3}{2}k_1(t)^2+\frac{a+b}{2},
  \end{aligned}
   \end{equation}
where $\epsilon_1,\epsilon_2 =\pm 1$, $h$ is any real constant, and $k_{(a,b)}$
is the Jacobi elliptic function
\begin{equation}\label{2.1.5}
 k_{(a,b)}(t)=
   \begin{cases}
     \sqrt{a}\cn\big(\sqrt{a-b} t,\frac{a}{a-b}\big),\quad  b<0,\\
     \sqrt{a}\dn\big(\sqrt{a} t,\frac{a-b}{a}\big), \quad\quad\,\,\,  b>0.
     \end{cases}
   \end{equation}
By possibly reversing the orientation along the curve, acting with an
orientation-reversing conformal transformation and shifting the independent
variable, we can assume that $\epsilon_1=\epsilon_2=1$ and $h=0$. We have
proved the following.

\begin{prop}\label{prop1}
The M\"obius classes of quasi-periodic critical curves are in 1-1 correspondence
with the points $(a,b)$ of the planar region
\begin{equation}\label{2.1.6}
 S=\left\{(a,b) \,:\, a>0,\, a>b,\, b\neq 0 \right\}.
  \end{equation}
 We call $a$ and $b$ the {\em parameters} of a quasi-periodic critical curve.
 \end{prop}

\begin{remark}
The M\"obius class of a critical curve with parameters $a$ and $b$ is represented by any conformal parametrization $\gamma :\R\to \R^3$ with curvatures
 \begin{equation}\label{2.1.7}
k_1=k_{(a,b)},\quad k_2=-\frac{3}{2}k_{(a,b)}^2+\frac{a+b}{2}.
\end{equation}
Note that the minimal period of $k_{(a,b)}$ is given by
\begin{equation}\label{2.1.8}
\omega_{(a,b)}=
 \begin{cases}
  {4K \left(a (a-b)^{-1}\right)}/{\sqrt{a-b}}
   \quad & b<0,\\
   {2K\left(b^{-1}(a-b)\right)}/{\sqrt{a}}
    \quad & b>0,
    \end{cases}
      \end{equation}
where
\begin{equation}\label{2.1.9}
K(m)=\int_0^{\pi/2}\frac{1}{\sqrt{1-m\sin^2t}}dt
  \end{equation}
is the complete elliptic integral of the first kind.
\end{remark}

\subsection{Closure conditions}\label{2.2}

Let $\Sigma$ denote the open subset of $S\subset \mathbb R^2$ defined by
\begin{equation}\label{Sigma}
  \Sigma=\left\{(a,b)\in S \,:\, ab>1 \right\}.
  \end{equation}
For every $(a,b)\in \Sigma$, let
\begin{equation}\label{2.2.1}
\begin{aligned}
 \mu(a,b)&=\frac{1}{\sqrt{2}}\sqrt{a+b+\sqrt{4+(a-b)^2}},\\
  \upsilon(a,b)& =\frac{1}{\sqrt{2}}\sqrt{a+b-\sqrt{4+(a-b)^2}},
  \end{aligned}
   \end{equation}
and define
\begin{equation}\label{2.2.2}
 \begin{aligned}
  \Phi_1(a,b)&:=\frac{1}{2\pi}\int_0^{\omega(a,b)}\frac{\mu(a,b)}{k_{(a,b)}(t)^2-\mu(a,b)^2}dt,\\
   \Phi_2(a,b)&:=\frac{1}{2\pi}\int_0^{\omega(a,b)}\frac{\upsilon(a,b)}{k_{(a,b)}(t)^2-\upsilon(a,b)^2}dt.
    \end{aligned}
\end{equation}
These integrals can be evaluated in term of the complete integral of the third kind (cf., for example, \cite{W})
\begin{equation}\label{thirdkind}
  \Pi(n,m)=\int_0^{\pi/2}\frac{dt}{(1-n\sin^2 t)\sqrt{1-m\sin^2 t}},\quad -1<n,m<1.
    \end{equation}
As a result, we have
\begin{equation}
\begin{aligned}\label{2.2.3}
 \Phi_1(a,b)&=\frac{\mu(a,b)}{\pi \sqrt{a}(a-\mu(a,b)^2)}\Pi\left(\frac{a-b}{a-\mu(a,b)^2},\frac{a-b}{a}\right),\\
 \Phi_2(a,b)&=\frac{\upsilon(a,b)}{\pi \sqrt{a}(a-\upsilon(a,b)^2)}\Pi\left(\frac{a-b}{a-\upsilon(a,b)^2},\frac{a-b}{a}\right).
  \end{aligned}
\end{equation}

\begin{defn}\label{def:period-map}
We call $\Phi = (\Phi_1,\Phi_2) : \Sigma \to \R^2$ the {\it period map}
of the conformal arclength functional.
By construction, $\Phi$ is non constant and real analytic.
\end{defn}

We are now in a position to state the following.

\begin{thm}\label{thm1}
 A quasi-periodic critical curve with parameters $a$ and $b$ is a conformal string if and only if
 $(a,b)\in \Sigma$ and $\Phi_1(a,b)$, $\Phi_2(a,b)\in \mathbb{Q}$.
  \end{thm}

For every $(a,b)\in S$, let $\gamma:\R\to \R^3$ be the conformal parametrization of a quasi-periodic
critical curve with parameters $(a,b)\in S$, such that $\mathrm{F}(0)=\mathrm{Id}$, where
$\mathrm{F}$ is the Vessiot frame along $\gamma$. The claim is that $\gamma$ is periodic if
and only if $(a,b)\in \Sigma$ and
$\Phi_j(a,b)\in \mathbb{Q}$, $j=1,2$. For brevity, let $k$ denote the
first conformal curvature $k_1$, let $\mu$ and $\upsilon$ denote the constants $\mu(a,b)$,
$\upsilon(a,b)$, respectively, and let $\omega$ denote the minimal period $\omega(a,b)$.

\begin{lemma}[cf. \cite{M1}]\label{l:momentum}
The canonical null lift $\Gamma :\R\to \mathcal{L}_+$ of $\gamma$ satisfies the linear system
\begin{align}\label{2.2.4}
  X'&=\left(\frac{k'}{k}\mathrm{Id}-\frac{k^2-a-b}{k'^2+1}\mathfrak{X}
   +\frac{k'(k^2-a-b)}{k(k'^2+1)}\mathfrak{X}^2 \right. \\
\notag     &\qquad \left.+\frac{1}{k'^2+1}\mathfrak{X}^3-\frac{k'}{k(k'^2+1)}\mathfrak{X}^4\right) X,
     \end{align}
with the initial condition $\Gamma(0)=e_4$, where
\[ \mathfrak{X}
=\left(
\begin{smallmatrix}
  0 & 0 & 1 & 0 & 0 \\
   1 & 0 & 0 & \sqrt{a} & 0 \\
    -b/2 & 0 & 0 & 0 & 1 \\
    0 & -\sqrt{a} & 0 & 0 & 0 \\
    0 & 1 & -b/2 & 0 & 0\\
    \end{smallmatrix}
    \right)\in \mathfrak{g}.
   \]
\end{lemma}

Using Lemma \ref{l:momentum}, an explicit integration of the critical curves was carried
out in \cite{M1} (see also \cite{MMR}).
The proof of Theorem \ref{thm1} is a direct consequence of this integration.
According to the analysis carried out in \cite{M1}, we are led to consider three cases,
depending on whether $ab<1$, $ab=1$, or $ab>1$.
It follows from \cite{M1} that the conformal parametrizations of a quasi-periodic critical
curve with parameters $a$, $b$ can only be periodic if $ab>1$, in which case the following holds.

\begin{lemma}\label{RL}
 A quasi-periodic critical curve with $ab>1$ is closed if and only if
  $\Phi_1(a,b)$ and $\Phi_2(a,b)$ are rational numbers.
   \end{lemma}

\begin{proof}[Proof of Lemma \ref{RL}]
Assume that $a>0$, $a>b$, and $ab>1$. Then, the eigenvalues of $\mathfrak{X}$ are
$\lambda_0=0$, $\lambda_{1,\pm}=\pm  i\mu$, and $\lambda_{2,\pm}= \pm  i \upsilon$.
Choose $Y\in GL(5,\C)$, such that
\[
\hat{\mathfrak{X}}=Y^{-1} \mathfrak{X} Y=\left(
      \begin{smallmatrix}
         0 & 0 & 0 & 0 & 0 \\
              0 & i\mu & 0 & 0 & 0 \\
               0 & 0 & -i\mu & 0 & 0 \\
             0 & 0 & 0 & i\upsilon & 0  \\
           0 & 0 & 0 & 0 & -i\upsilon \\
         \end{smallmatrix}
   \right)
     \]
and consider the map $V=Y^{-1} \Gamma:\R\to \C^5$.
From \eqref{2.2.4}, it follows that the components $v_0,\dots,v_4$ of $V$ satisfy
\[
\begin{cases}
v_0'=\frac{k'}{k}v_0,\\
v'_1=\left( \frac{k'}{k}-i\frac{k^2-a-b}{k'^2+1}\mu - \frac{k'(k^2-a-b)}{k(k'^2+1)}\mu^2
-i \frac{1}{k'^2+1}\mu^3+\frac{k'}{k(k'^2+1)}\mu^4\right)v_1,\\
v'_2=\left( \frac{k'}{k}+i\frac{k^2-a-b}{k'^2+1}\mu - \frac{k'(k^2-a-b)}{k(k'^2+1)}\mu^2
+i \frac{1}{k'^2+1}\mu^3+\frac{k'}{k(k'^2+1)}\mu^4\right)v_2,\\
v'_3=\left( \frac{k'}{k}-i\frac{k^2-a-b}{k'^2+1}\upsilon - \frac{k'(k^2-a-b)}{k(k'^2+1)}\upsilon^2
-i \frac{1}{k'^2+1}\upsilon^3+\frac{k'}{k(k'^2+1)}\upsilon^4\right)v_3,\\
v'_4=\left( \frac{k'}{k}+i\frac{k^2-a-b}{k'^2+1}\upsilon - \frac{k'(k^2-a-b)}{k(k'^2+1)}\upsilon^2
+i \frac{1}{k'^2+1}\upsilon^3+\frac{k'}{k(k'^2+1)}\upsilon^4\right)v_4.
\end{cases}
\]
Using $k'^2+(k^2-a)(k^2-b)=0$, the above equations take the form
\begin{align*}
 v'_1&=-\frac{kk'+i\mu}{\mu^2-k^2}v_1,\quad
  v_2'=\frac{-kk'+i\mu}{\mu^2-k^2}v_2,\\
    v_3'&=-\frac{kk'+i\upsilon}{\upsilon^2-k^2}v_3,\quad
     v_4'=\frac{-kk'+i\upsilon}{\upsilon^2-k^2}v_4.
\end{align*}
Thus $v_a = p_a w_a$, where $p_0,\dots,p_4\in \C$ and $w_0,\dots,w_4$ are the complex-valued functions
given by
\[
\begin{aligned}
w_0&=k,
\quad w_1=\sqrt{\mu^2-k^2}e^{i\int \frac{\mu}{\mu^2-k^2}dt},
\quad w_2=\sqrt{\mu^2-k^2}e^{-i\int \frac{\mu}{\mu^2-k^2}dt},\\
w_3&=\sqrt{\upsilon^2-k^2}e^{i\int \frac{\upsilon}{\upsilon^2-k^2}dt},\quad
w_4 =\sqrt{\upsilon^2-k^2}e^{-i\int \frac{\upsilon}{\upsilon^2-k^2}dt}.
\end{aligned}
\]
Since $k$ is a non-constant periodic function with minimal period $\omega$,
the functions $w_0,\dots,w_4$ are periodic if and only if
\[
  \Phi_1=\frac{1}{2\pi}\int_0^{\omega}\frac{\mu}{k(t)^2-\mu^2}dt\in \mathbb{Q},
   \quad  \Phi_2=\frac{1}{2\pi}\int_0^{\omega}\frac{\upsilon}{k(t)^2-\upsilon^2}dt\in \mathbb{Q},
   \]
which proves the lemma.
\end{proof}


This proves Theorem \ref{thm1}. As a corollary, we have the following.

\begin{prop}\label{prop2}
The M\"obius classes of conformal strings are in 1-1 correspondence
with the elements of the countable set
\[
 \Sigma_*=\left\{(a,b)\in \R^2 \,:\, a>0,\, a>b,\, ab>1,\, \Phi_1(a,b),\, \Phi_2(a,b)\in \mathbb{Q}\right\}.
  \]
   \end{prop}

   This Proposition will be used in the proof of Theorem \ref{thm:A}.

\section{The proof of Theorem \ref{thm:A}}\label{s:3}

Theorem \ref{thm:A} asserts that the M\"obius classes of
conformal strings are in one-to-one correspondence with the rational points of the
complex domain
\[
\Omega =
  \left\{q\in \mathbb{C} \,:\, {1}/{2} < \Re q< {1}/{\sqrt{2}},\,\,
   \Im q>0,\,\, |q| < {1}/{\sqrt{2}}\right\}.
    \]

The proof of Theorem \ref{thm:A} will follow from Proposition \ref{prop2}
and the next result about the period map $\Phi$ (cf. Definition \ref{def:period-map}).

\begin{thm}\label{thm2}
The period map $\Phi=(\Phi_1,\Phi_2) : \Sigma \to \R^2$ is a real-analytic diffeomorphism onto the domain
\[
 \widetilde{\Omega}=\left\{(x,y)\in \R^2 \,:\, -1/\sqrt{2}<x<-1/2,\, x^2+y^2<1/2,\, y>0 \right\}.
  \]
\end{thm}

The proof of Theorem \ref{thm2}
is based on a detailed study of the analytic properties
of $\Phi$. This study is split into several technical results
which will take up the remaining part of the section.


\subsection{Preparatory material}

Note that $\Phi_1$, $\Phi_2 : \Sigma \to \mathbb R$, and hence the period map
$\Phi =(\Phi_1,\Phi_2)$, extend analytically to the domain
\[
 \widetilde{\Sigma}=\left\{(a,b) \,:\, a>0,\, ab \ge 1 \right\}.
  \]
Slightly abusing notation, we will retain the same letters for
the extensions of $\Phi_1$ and $\Phi_2$ to $\widetilde{\Sigma}$.
Let
\begin{equation}\label{4.1}
 E:m\in [0,1)\mapsto \int_0^{\pi/2}\sqrt{1-m\sin^2t} dt\in \R
  \end{equation}
be the elliptic integral of the second kind. Then, $E(m)<K(m)$
and the ratio $E/K$ is a strictly decreasing function from $[0,1)$
onto $(0,1]$. The power series of $K$ and $E$
(cf. \cite{La}, p. 73, and \cite{W}) are given by
\begin{equation}\label{4.2}
 \begin{cases}
   K(m)=\frac{\pi}{2}\left(1+\frac{1}{4}m+\left(\frac{1\cdot 3}{2\cdot 4}\right)^2m^2
       +\left(\frac{1\cdot 3\cdot 5}{2       \cdot 4\cdot 6}\right)^2 m^3+\cdots \right),\\
     E(m)=\frac{\pi}{2}\left(1-\frac{1}{4}m-\frac{1}{3}\left(\frac{1\cdot 3}{2\cdot 4}\right)^2m^2
       -\frac{1}{5}\left(\frac{1\cdot 3\cdot 5}{2\cdot 4\cdot 6}\right)^2 m^3 +\cdots\right).
        \end{cases}
\end{equation}
We also recall the expressions of the derivatives of the complete elliptic integral
 of the third kind $\Pi$ given in \eqref{thirdkind} (cf. \cite{La}, \S 3.7-3.9, and \cite{W}),
\begin{equation}\label{4.3}
\begin{cases}
 \partial_n\Pi|_{(n,m)} = \frac{nE(m)+(m-n)K(m)+(n^2-m)\Pi(n,m)}{2(m-n)(n-1)n},\\
   \partial_m\Pi|_{(n,m)} = \frac{E(m)}{2(m-1)(n-m)}+\frac{\Pi(n,m)}{2(n-m)}.
     \end{cases}
      \end{equation}

\subsection{The derivatives of $\Phi_1$ and $\Phi_2$}

Using \eqref{2.2.3} and \eqref{4.3}, we have
\begin{equation}\label{4.4}
\begin{cases}
(\mathrm i) & \partial_a\Phi_1|_{(a,b)}=\frac{X_{11}(a,b)E(\frac{a-b}{a})
    +Y_{11}(a,b)K(\frac{a-b}{b})}{Z_{11}(a,b)},\\
(\mathrm {ii})&\partial_b\Phi_1|_{(a,b)}=\frac{X_{21}(a,b)E(\frac{a-b}{a})
    +Y_{21}(a,b)K(\frac{a-b}{b})}{Z_{21}(a,b)},\\
(\mathrm {iii})&\partial_a\Phi_2|_{(a,b)}=\frac{X_{12}(a,b)E(\frac{a-b}{a})
       +Y_{12}(a,b)K(\frac{a-b}{b})}{Z_{12}(a,b)},\\
(\mathrm {iv})&\partial_b\Phi_2|_{(a,b)}=\frac{X_{22}(a,b)E(\frac{a-b}{a})
                     +Y_{22}(a,b)K(\frac{a-b}{b})}{Z_{22}(a,b)},\\
\end{cases}
\end{equation}
where the coefficients $X_{ij}(a,b)$, $Y_{ij}(a,b)$, and $Z_{ij}(a,b)$ are given by
\[
\begin{cases}
X_{11}(a,b)=\sqrt{2}\left(2 z(a,b)-a^2b+a\left(4+ z(a,b)b\right)+b\left(
4+ z(a,b)b\right)\right),\\
Y_{11}(a,b)=-2\sqrt{2}\left(a+ z(a,b)-ab^2 +b\left(3+b( z(a,b)+b) \right) \right),\\
Z_{11}(a,b)=\pi\sqrt{a} z(a,b)(a-b)(a-b- z(a,b))\left(a+b+ z(a,b) \right)^{3/2},\\
X_{21}(a,b)=a(2b+ z(a,b)),\\
Y_{21}(a,b)=-b(a+b+ z(a,b)),
\end{cases}
\]
and by
\[
\begin{cases}
Z_{21}(a,b)=\sqrt{2}\pi b(a-b) z(a,b)\sqrt{a(a+b+ z(a,b))},\\
X_{12}(a,b)=\sqrt{2}\left(2 z(a,b)+a^2b+a(-4+b z(a,b))-b(4-b z(a,b)+b^2)\right),\\
Y_{12}(a,b)=2\sqrt{2}\left( a- z(a,b)-ab^2+b(3-b z(a,b)+b^2)\right),\\
Z_{12}(a,b)=\pi\sqrt{a} z(a,b)(a-b)\left(a-b+ z(a,b)\right)\left( a+b- z(a,b)\right)^{3/2},\\
X_{22}(a,b)=a\left( z(a,b)-2b\right),\\
Y_{22}(a,b)=b\left(a+b- z(a,b)\right),\\
Z_{22}(a,b)=\sqrt{2}\pi(a-b)b z(a,b)\sqrt{a\left(a+b- z(a,b) \right)}.
\end{cases}
\]
In the formulae above, $z(a,b)$ stands for $\sqrt{4+(a-b)^2}$. These formulae have been
derived with the help of the software {\sl Mathematica}. We now prove the following.

\begin{prop}\label{Prop4.I}
The partial derivatives of $\Phi_1$ and $\Phi_2$ are strictly positive on
\begin{equation}\label{4.5}
 \Sigma'=\left\{(a,b)\,:\, a>1,\, ab>1,\, b\le a\right\}.
   \end{equation}
\end{prop}

\begin{proof}
We begin by proving the following.

\begin{lemma}\label{lemma4I}
$\partial_a\Phi_1|_{(a,b)}>0$, for every $(a,b)\in \Sigma'$.
\end{lemma}

\begin{proof}[Proof of Lemma \ref{lemma4I}]
Consider $a>1$ and $b\in (0,a]$. Let $b=(1-m)a$, with $m\in [0,1)$, and let
$\tilde{z}(a,m):=\sqrt{4+m^2a^2}$.
From \eqref{4.4}$(\mathrm {i})$, we have
\[
\partial_a\Phi_1|_{(a,(1-m)a)}=\frac{\tilde{X}_{11}(a,m)E(m)
  +\tilde{Y}_{11}(a,m)K(m)}{\tilde{Z}_{11}(a,m)},
  \]
where
\[
\begin{cases}
\tilde{X}_{11}(a,m)=\sqrt{2}\left(-2\tilde{z}(a,m)+(2-m)a \left((1-m)a(ma-\tilde{z}(a,m))-4)\right) \right),\\
\tilde{Y}_{11}(a,m)=2\sqrt{2}\left((4\!-\!3m)a\!-\!(1\!-\!m)^2ma^3\!+\!\tilde{z}(a,m)
\!+\!(1\!-\!m)^2a^2\tilde{z}(a,m) \right),\\
\tilde{Z}_{11}(a,m)=\pi ma\tilde{z}(a,m)(\tilde{z}(a,m)-am)\left((2-m)a+\tilde{z}(a,m)\right)^{3/2}.
\end{cases}
\]

First, we prove that $\partial_a\Phi_1|_{(a,(1-m)a)}>0$, for every $m\in (0,1)$.
%
%
 If $m\in (0,1)$, then $\tilde{Z}_{11}(a,m)>0$. So, if we set $f_m(a)=\tilde{X}_{11}(a,m)/\tilde{Y}_{11}(a,m)$,
it suffices to show that $f_m(a)+K(m)/E(m)>0$. The derivative of $f_m(a)$
with respect to $a$ is
\[
  f_m'(a)=- \frac{2m}{4(2\!-\!m)a\!+\!(2\!-\!m)m^2a^3\!+\!2\tilde{z}(a,m)\!+\!(2\!-\!(2\!-\!m)m)a^2\tilde{z}(a,m)}.
   \]
Thus, $f_m$ is a strictly decreasing function such that
$\lim_{a\to \infty} f_m(a) = -(2+m)/2$. This implies that
\[
  f_m(a)+\frac{K(m)}{E(m)}>\lim_{a\to \infty} f_m(a) + \frac{K(m)}{E(m)} = \frac{K(m)}{E(m)}-\frac{2+m}{2}.
   \]
Deriving the function $\tilde{f}(m)=K(m) E(m)^{-1}-(2+m)/2$, we find
\[
 \tilde{f}'|_m = \frac{(E(m)-K(m))^2}{2mE(m)^2}+\frac{m E(m)^2}{2(1-m)E(m)^2}>0.
  \]
Then, $\tilde{f}$ is strictly increasing and $\tilde{f}(m)>\lim_{m\to 0}\tilde{f}=0$,
for every $m\in (0,1)$.
This implies $\partial_a\Phi_1|_{(a,(1-m)a)}>0$, for every $m\in (0,1)$.
\vskip0.1cm

Next, we prove that $\partial_a\Phi_1|_{(a,a)}>0$.
%
Using the power series expansions of $K(m)$ and $E(m)$, we find
\[
\begin{cases}
 \tilde{X}_{11}(a,m)E(m)+\tilde{Y}_{11}(a,m)K(m)=\frac{\pi}{\sqrt{2}}(2a+ \tilde{z}(a,m))m+ o(m),\\
 \tilde{Z}_{11}(a,m)=\pi a \sqrt{a}\tilde{z}(a,m)^2((2-m)a+\tilde{z}(a,m))^{3/2}m+o(m).
  \end{cases}
    \]
From this, we have
\[
  \partial_a\Phi_1|_{(a,a)}=\lim_{m\to 0}\partial_a\Phi_1|_{(a,(1-m)a)} = \frac{1}{8a^{3/2}\sqrt{1+a}}>0.
    \]
\end{proof}

By similar arguments as in the proof of Lemma \ref{lemma4I}, one can prove that,
for every $(a,b)\in \Sigma'$, $\partial_b\Phi_1|_{(a,b)}>0$,
$\partial_a\Phi_2|_{(a,b)}>0$, and $\partial_b\Phi_2|_{(a,b)}>0$,
which completes the proof of Proposition \ref{Prop4.I}.
\end{proof}

\subsection{The image of $\Phi$}

Let $\Sigma$ be as in \eqref{Sigma}. We will now prove the following.

\begin{prop}\label{prop4.II}
The Jacobian of $\Phi$ is strictly positive on $\Sigma$.
In particular, the image $\Phi(\Sigma)$ is a connected open set and $\Phi:\Sigma\to
\Phi(\Sigma)$ is a local diffeomorphism.
\end{prop}

\begin{proof}
Let $\Phi_*$ be the matrix of the differential of $\Phi$. Using \eqref{4.4},
the Jacobian $J\Phi(a,b)$ of $\Phi$ at $(a,b)$ is given by
\[
 J\Phi(a,b) =\mathrm{det}(\Phi_*)|_{(a,b)}=\frac{E(\frac{a-b}{b})\left(-2aE(\frac{b-a}{a})
         +(a+b)K(\frac{a-b}{a}) \right)}{2\pi^2a(a-b)b  \sqrt{(4+(a-b)^2)(ab-1)}}.
      \]
Let $w(a,b)$ denote the numerator of the right hand side.
If we let $b=(1-m)a$, where $m\in (0,1)$ and $a>1/\sqrt{1-m^2}$, then
\[
  w(a,(1-m)a)=a(-2E(m)+(2-m)K(m))>0,\quad \forall\, m\in (0,1),
    \]
which implies the required result.
\end{proof}

We adopt the following conventions:
\begin{itemize}
\item $\Phi=(\Phi_1,\Phi_2)$ is considered as a function defined on $\Sigma$;
\item $\widetilde{\Phi}$ denotes the analytic extension of $\Phi$ to $\widetilde{\Sigma}=\{(a,b) \,:\, a>0,\, ab \ge 1\}$;
\item $\widehat{\Phi}$  denotes the restriction of $\tilde{\Phi}$ to the closure $\bar{\Sigma}$ of $\Sigma$.
\end{itemize}

Next, we prove the following.

\begin{prop}\label{prop4.III}
The mapping $\Phi$ is a real-analytic local diffeomorphism onto
$\widetilde{\Omega}$.
\end{prop}

\begin{proof}
By Proposition \ref{prop4.II}, it suffices to prove
that $\widetilde{\Omega}$ is the image of $\Phi$. The boundary of $\Sigma$ consists of the simple arcs
\[
  \partial_+\Sigma=\{(a,a) \,:\, a\ge 1\},\quad \partial_{-}\Sigma =\{(a,1/a) \,:\, a\ge 1\}
    \]
which intersect at the point $V=(1,1)$, while the boundary of $\widetilde{\Omega}$
is made of the three simple arcs
\[
\begin{aligned}
& \partial_{-}\widetilde{\Omega}=\{(a,0) \,:\, a\in [-1/\sqrt{2},-1/2]\},\\
& \partial_0 \widetilde{\Omega} =\{(-1/2,b) \,:\, b\in [0,1/2]\},\\
&\partial_{+}\widetilde{\Omega} =\{(a,\sqrt{a^2-1/2})\,:\, a\in [-1/\sqrt{2},-1/2]\}
\end{aligned}
\]
with vertices $Q_1=(-1/\sqrt{2},0)$, $Q_2=(-1/2,0)$ and $Q_3=(-1/2,-1/2)$, respectively.
The restriction $ \Phi_{+}$ of $\widehat{\Phi}$ to $\partial_+\Sigma$ is given by
\[
 \Phi_{+}: (a,a) \in \partial_+\Sigma \mapsto \frac{1}{2} \left(-\sqrt{(a+1)a^{-1}},
            \sqrt{(a-1)a^{-1}}\right)
   \in \partial_{+}   \widetilde{\Omega}\setminus\{Q_3\}.
    \]
%
%
It is then a diffeomorphism of $\partial_+\Sigma $ onto $\partial_{+}\widetilde{\Omega}\setminus\{Q_3\}$.
Similarly, the restriction $\Phi_{-}$ of $\widehat{\Phi}$ to $\partial_{-}\Sigma$ is
the diffeomorphism onto $\partial_{-}\widetilde{\Omega}\setminus\{Q_2\}$ given by
\[
  \Phi_{-}:(a,1/a)\in \partial_{-}\Sigma \mapsto \left(\!-\frac{1}{\pi}\sqrt{1\!+\!a^2}\Pi(1\!-\!a^2,(a^2\!-\!1)a^{-2}), 0\right)\!\in
   \partial_{-}\widetilde{\Omega}\setminus\{Q_2\}.
    \]
 Then $\widehat{\Phi}$ maps the boundary of $\Sigma$ to the boundary of $\widetilde{\Omega}$.
Next, we show that the image of $\Phi$ is contained in $\widetilde{\Omega}$.
For, fix $a>1$ and consider the curve defined by
$\phi_a (b)=\Phi(a,b)$, for every  $b\in (1/a,a)$. From Proposition \ref{Prop4.I},
 we know that the components $\phi^1_a$ and $\phi^2_a$ of $\phi_a$ are increasing functions.
 This implies
\[
 \begin{split}
  -\frac{1}{\sqrt{2}}  <\phi_a^1(a^{-1})&=-\pi^{-1} \sqrt{1+a^2}\Pi(1-a^2,(a^2-1)  a^{-2})\\
&<\phi_a^1(b)<\phi_a^1(a)= -\frac{1}{2}\sqrt{(a+1) a^{-1}}<-\frac{1}{2},
\end{split}
\]
and
\[
 0=\phi_a^2(1/a)<\phi_a^2(b)<\phi_a^2(a)=\frac{1}{2}\sqrt{(a-1)  a^{-1}}.
 \]
  Combining these two inequalities, we find
\[
 0<\phi_a^1(b)^2+\phi_a^2(b)^2<\phi_a^1(a)^2+\phi_a^2(a)^2=1/2.
  \]
This shows that the image of $\Phi$ is contained in $\widetilde{\Omega}$. To prove equality, we need
the following technical lemma.

\begin{lemma}\label{4.V}
Let $\{(a_n,b_n)\}_{n\in \mathbb{N}}\subset \Sigma$ be a sequence such that $a_n\to \infty$. Then, $\lim_{n\to \infty} \Phi_1(a_n,b_n)=-1/2$.
\end{lemma}

 \begin{proof}
 The trajectory of $\phi_a$ is the graph of an increasing function
\[
 f_a:(\Phi_1(a,a^{-1}),\Phi_1(a,a))\to \R.
  \]
Given the sequence $\{(a_n,b_n)\}_{n\in \mathbb{N}}$, we write $\Phi(a_n,b_n)=(t_n, f_{a_n}(t_n))$,
where $t_n$ is an element of the open interval $(\Phi_1(a_n,1/a_n),\Phi_1(a_n,a_n))$.
From the limits
\[
\begin{aligned}
 \lim_{a\to \infty} \widehat{\Phi}_1(a,a^{-1}) &= - \lim_{a\to \infty} \frac{1}{\pi}\sqrt{1+a^2}\Pi(1-a^2,(a^2-1)a^{-2}) = -\frac{1}{2},\\
 \lim_{a\to \infty} \widehat{\Phi}_1(a,a) &= - \lim_{a\to \infty} \frac{1}{2}\sqrt{(1+a)a^{-1}}=-\frac{1}{2},
\end{aligned}
\]
we have
\[
 -\frac{1}{2} = \lim_{n\to \infty} t_n = \lim_{n\to \infty}\Phi_1(a_n,b_n).
   \]
\end{proof}
By Proposition \ref{prop4.II}, $\Phi:\Sigma\to \widetilde{\Omega}$
is a real-analytic local diffeomorphism onto its image, and hence an open map.
Consequently, $X=\widetilde{\Omega}\setminus\Phi(\Sigma)$ is a closed subset of $\widetilde{\Omega}$.
The proof is complete if we show that $\widetilde{\Omega}\setminus\Phi(\Sigma)$ is also open.
Suppose it is not open. Then, there is a point $Q\in X$, such that any open disk $D(Q,1/n)$ centered at $Q$
with radius $1/n$, $n\in \mathbb{N}$, intersects $\Phi(\Sigma)$. For every $n\in \mathbb{N}$,
choose $Q_n\in D(p,1/n)\cap \Phi(\Sigma)$ and $P_n\in \Sigma$, such that $\Phi(P_n)=Q_n$.
Two possibilities may occur: either $\{P_n\}$ is bounded, or else $\{P_n\}$ is unbounded.
In the first case, we may assume that $\{P_n\}$ converges to a limit point $P$. By construction,
$P$ belongs to $\bar{\Sigma}$. If $P$ is an element of $\Sigma$, then $Q=\Phi(P)$,
which contradicts the fact that
$Q\notin \Phi(\Sigma)$.
So, $P$ is an element of the boundary $\partial \Sigma$. This implies that $Q=\widehat{\Phi}(P)$.
On the other hand, $\widehat{\Phi}$ maps the boundary of $\Sigma$ onto the boundary of $\widetilde{\Omega}$. Consequently, $Q$ would be an element of $\partial \widetilde{\Omega}$, which is absurd, since $Q\in \widetilde{\Omega}$. If  $\{P_n\}$ is unbounded, by Lemma \ref{4.V}, the sequence made up with the
abscissae of the points $Q_n=\Phi(P_n)$ converges to $-1/2$. So, the first coordinate of $Q$
would be equal to $-1/2$ and hence $Q\notin \widetilde{\Omega}$, which is a contradiction.
This concludes the proof of Proposition \ref{prop4.III}.
\end{proof}

\subsection{Injectivity of $\Phi$}

By Proposition \ref{Prop4.I}, we know that the first order partial derivatives of $\Phi_1$ and $\Phi_2$
are strictly positive on $\Sigma'$. Then, there is an open neighborhood $W$ of $\Sigma'$ such that
the first order partial derivatives of $\Phi_1$ and $\Phi_2$ are positive on
$W'=W\cap \mathrm{Int}(\widetilde{\Sigma})$.  On this set, consider the nowhere
vanishing vector fields (see Figure \ref{FIG2})
\begin{equation}\label{4.6}
U_1=\left(1,-\frac{\partial_a \Phi_1}{\partial_b\Phi_1}\right),\quad
U_2=\left(1,-\frac{\partial_a\Phi_2}{\partial_b\Phi_2}\right).
\end{equation}
\begin{figure}[ht]
\begin{center}
\includegraphics[height=6cm,width=6cm]{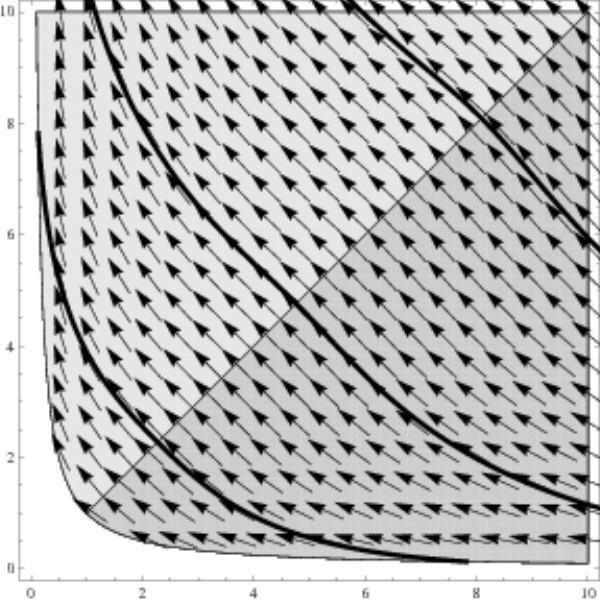}
\includegraphics[height=6cm,width=6cm]{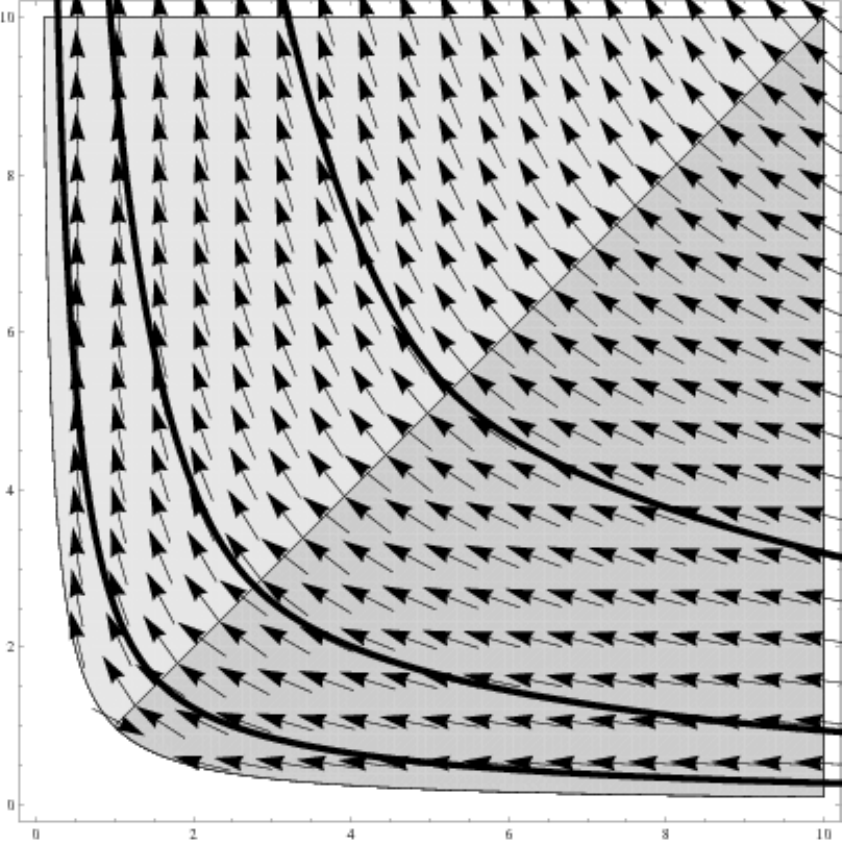}
\caption{The vector fields $U_1/\|U_1\|$ and $U_2/\|U_2\|$, with some
of their integral curves.}\label{FIG2}
\end{center}
\end{figure}
 By construction, the trajectories of the integral curves of $U_1$ and $U_2$ are graphs of strictly decreasing functions, and hence they intersect $\partial_+ \Sigma$ in at most one point.
 In addition, for every  $Q=(x,y)\in \widetilde{\Omega}$, we have
\begin{itemize}
\item the connected components of the level curve $\mathcal{V}_1(x)=\Phi_1^{-1}(x)\cap \Sigma$ are contained in the intersection of a trajectory of $U_1$ with $\Sigma$.
\item the connected components of the level curve $\mathcal{V}_2(y)=\Phi_2^{-1}(y)\cap \Sigma$ are contained in the intersection of a trajectory of $U_2$ with $\Sigma$.
\end{itemize}

We now prove the following.
\begin{lemma}\label{4.VI}
The level curve $\mathcal{V}_1(x')$ is connected, for every $Q'=(x',y')\in \widetilde{\Omega}$.
\end{lemma}

\begin{proof}
Consider the arcs $\partial_+\widetilde{\Omega}'=\partial_+\widetilde{\Omega}\setminus\{Q_1,Q_3\}$,
$\partial_-\widetilde{\Omega}'=\partial_-\widetilde{\Omega}\setminus\{Q_1,Q_2\}$,
$\partial_+ \Sigma'\!=\!\partial_+\Sigma \setminus\!\{V\}$ and $\partial_- \Sigma'\!=\!\partial_-\Sigma \setminus\!\{V\}$.
From the proof of Proposition~\ref{prop4.III}, we know that the maps
$\widehat{\Phi}|_{\partial_+ \Sigma'}:\partial_+ \Sigma'\to \partial_+\Omega'$ and
$\widehat{\Phi}|_{\partial_- \Sigma'}:\partial_- \Sigma'\to \partial_-\Omega'$
are real-analytic diffeomorphisms. For every $Q'=(x',y')\in \widetilde{\Omega}$, let $Q'_{\pm}$
denote the intersections of $\partial_{\pm}\widetilde{\Omega}'$ with the line $x=x'$ and let
$P'_+=(a'_+,a'_+)$ and $P'_- = (a'_-, 1/a'_{-})$ be the points of $\partial_{\pm} \Sigma'$,
such that $\widehat{\Phi}(P'_{\pm})=Q'_{\pm}$. Note that
$\widehat{\Phi}_1^{-1}(x')=\Phi_1^{-1}(x')\cup\{P'_+,P'_-\}$.
\vskip0.1cm

\noindent $\bullet$ First, we prove that $\widehat{\Phi}_1^{-1}(x')$ is compact.
By continuity, $\widehat{\Phi}_1^{-1}(x')$ is closed. Suppose that $\widehat{\Phi}_1^{-1}(x')$ is unbounded.
Then, there is a sequence $\{P_n\}=\{(a_n,b_n)\}\subset \Phi_1^{-1}(x')$, such that
$\lim_{n\to \infty}a_n=+\infty$. By Lemma \ref{4.V}, $\{\Phi_1(P_n)\}$ converges to $-1/2$.
On the other hand, the points $P_n$ belong to $\mathcal{V}_1(x')$, and hence $x'$, which is the limit of the sequence $\{\Phi_1(P_n)\}$, must be equal to
$-1/2$. But this is absurd, since $(x',y')\in \widetilde{\Omega}$ and the abscissae of the points in
$\widetilde{\Omega}$ are different from $-1/2$. This proves the compactness of $\widehat{\Phi}_1^{-1}(x')$.
\vskip0.1cm

\noindent $\bullet$ Next, we prove that the points $P'_{\pm}$ belong to the boundary of
any connected component of the level curve $\mathcal{V}_1(x')$. Let $C$
be a connected component of $\mathcal{V}_1(x')$. We know that $C$ is the graph of a strictly
decreasing function and that such a function is bounded, by the compactness of $\widehat{\Phi}_1^{-1}(x')$.
Since $C$ is an embedded curve, its
boundary
consists of two distinct points. On the other hand,
$C\subset \Phi_1^{-1}(x')\subset \widehat{\Phi}_1^{-1}(x')=\Phi_1^{-1}(x')\cup\{P'_+,P'_- \}$.
Using the fact that $\Phi_1^{-1}(x')$ is an embedded curve, we have that $P'_+$ and $P'_-$ are the
two boundary points of $C$.
\vskip0.1cm

\noindent $\bullet$ To complete the proof, we show that two connected components of the level curves
intersect each other.
The vector field $U_1$ does not vanishes at the point $P'_+$ and $P'_+$ belongs to the
boundary of $C$. Therefore, $C$ is contained in the trajectory of $U_1$ passing through $P'_+$.
Choose local coordinates $(u,v)$ defined on a cubical open neighborhood $A$ of the point
$P'_+$, such that $U_1|_{A}=\partial_u$ and $v(P'_+)=0$. In these coordinates, the intersection
$C\cap A$ consists of the points $P\in A$ such that $v(P)=0$ and $0<u(P)<\epsilon$, for some
$\epsilon>0$. So, if $C$ and $C'$ are two connected components of $\mathcal{V}_1(x')$,
there exist $\epsilon,\epsilon'>0$, such that $C\cap A = \{P\in A \,:\, v(P)=0,\, 0<u(P)<\epsilon\}$
and  $C'\cap A = \{P\in A \,:\, v(P)=0,\, 0<u(P)<\epsilon'\}$.
This implies that $C \cap C'\neq \emptyset$.
\end{proof}

A similar result also holds for the other family of level curves.

\begin{lemma}
The level curves $\mathcal{V}_2(y')$ are connected, for every $Q'=(x',y')\in \widetilde{\Omega}$.
\end{lemma}

 \begin{proof}
 Let $C$ be any connected component of $\mathcal{V}_2(y')$. Fix $(a,b)\in C$
 and take a strictly decreasing function $u : I \to \R$ whose graph coincides with $C$.
 Such a function is defined on an open interval $I$ of the form $(a-\epsilon_1,a+\epsilon_2)$,
 where $\epsilon_1,\epsilon_2>0$ and  $a-\epsilon_1>1$. Let $\{a_n\}\subset (a-\epsilon_1,a)$
 be a decreasing sequence, such that $\lim_{n\to \infty}a_n=a^*=a-\epsilon_1$.
By construction, $\{u(a_n)\}$ is an increasing sequence and $(a_n,u(a_n))\in C\subset \Sigma$,
 for every $n$. In particular,
$0<a_n^{-1}<u(a_n)<a_n<a$, which implies that $\{u(a_n)\}$ is bounded. By possibly
taking a subsequence, we may assume that $\{u(a_n)\}$ converges to a limit point $b^*$.
Since $\{(a_n,u(a_n))\}$ converges to $(a^*,b^*)$, the point $(a^*,b^*)$ belongs to
$\partial C \cap \bar{\Sigma}$.
\vskip0.1cm

\noindent $\bullet$ We prove that $(a^*,b^*)$ is an element of $\partial\Sigma$.
By contradiction, suppose that $(a^*,b^*)\in \Sigma$. Then, $C$ is
contained in the trajectory, $C'$, of the vector field $U_2|_{\Sigma}$ passing
through $(a^*,b^*)$. By definition, $(a^*,b^*)\in C'$ and $(a^*,b^*)\notin C$.
Therefore, $C'$ is a connected subset of $\mathcal{V}_2(y')$ which contains
properly $C$. This is absurd since $C$ is a connected component
of $\mathcal{V}_2(y')$.
\vskip0.1cm

\noindent By construction, since $1/a^*<u(a_n)$ and ${u(a_n)}$ is an increasing sequence,
$b^*\neq 1/a^*$. This shows that $(a^*,b^*)$ does not belong to $\partial_-\Sigma$.
Therefore, $(a^*,b^*)$ is an element of $\partial_+\Sigma$ and hence
$b^*=a^*$ and $a^*>1$.
The point $(a^*,a^*)$ is independent of the choice of the connected component.
Otherwise, the restriction of $\widetilde{\Phi}$ to $\partial_+\Sigma$ could not
be injective. We are now in the same situation as in the last part of
the proof of the previous lemma, i.e., the point $(a^*,a^*)$ belongs to the boundary
of any connected component of the level curve $\mathcal{V}_2(y')$.
By arguing as above, we deduce the result.
\end{proof}

 We are now ready to prove the last ingredient for the proof of Theorem~\ref{thm2}.

\begin{prop}\label{prop4.IV}
The mapping $\Phi$ is injective.
\end{prop}

\begin{proof}
It suffices to show that $\#\left(\mathcal{V}_1(x)\cap \mathcal{V}_2(y)\right)=1$,
for every $(x,y)\in \widetilde{\Omega}$. By contradiction, suppose the existence of
$(x,y)\in \widetilde{\Omega}$, such that $\sharp(\mathcal{V}_1(x)\cap \mathcal{V}_2(y))>1$.
Let $(a,b)$ and $(a_1,b_1)$ be two distinct elements of $\Sigma$, such that
$\Phi(a_1,b_1)$ $=$ $\Phi(a,b)$ $=$ $(x,y)$. From the above discussions, we know that the
level curves $\mathcal{V}_1(x)$ and $\mathcal{V}_2(y)$ are connected and
graphs of two strictly decreasing functions, denoted by $u$ and $v$, respectively.
The domain of definition is an open interval $I\subset (1,+\infty)$, containing $a$ and $a_1$.
By construction, $u(a)=v(a)=b$, $u(a_1)=v(a_1)=b_1$, with $a\neq a_1$.
On the other hand, $\mathcal{V}_1(x)$ and $\mathcal{V}_2(y)$ are contained in the
trajectories of the vector fields $U_1$ and $U_2$, respectively. From this, we have
\[
   u'(t)=-\frac{\partial_a\Phi_1}{\partial_b\Phi_1}\bigg|_{(t,u(t))},\quad
    v'(t)=-\frac{\partial_a\Phi_2}{\partial_b\Phi_2}\bigg|_{(t,v(t))},\quad \forall t\in I.
      \]
Now, Proposition \ref{Prop4.I} and Proposition \ref{prop4.II} imply
\begin{equation}
\frac{\partial_a\Phi_1}{\partial_b\Phi_1}\bigg|_{(\alpha,\beta)}-
\frac{\partial_a\Phi_2}{\partial_b\Phi_2}\bigg|_{(\alpha,\beta)}>0,
\quad \forall (\alpha,\beta)\in \Sigma.
\end{equation}
Then, the function $h=v-u$ satisfies $h(a)=h(a_1)=0$, $h'(a)>0$ and $h'(a_1)>0$.
This implies the existence of
$a_2\in I$, different from $a$ and $a_1$, such that $h(a_2)=0$ and $h'(a_2)\le 0$.
Consequently, the Jacobian of $\Phi$ is non positive at
$(a_2,b_2)$ $=$ $(a_2,u(a_2))=(a_2,v(a_2))\in \Sigma$,
contrary to Proposition~\ref{prop4.II}.
\end{proof}

\section{The proof of Theorem \ref{thm:B}}\label{s:4}

Theorem \ref{thm:B} asserts that
any M\"obius class of conformal strings is represented by a model conformal string.
We will begin the proof by describing the model strings.

\subsection{The symmetrical configuration of a string}

We have just proved that the M\"obius classes of conformal strings are in
one-to-one correspondence with the {\it moduli}, i.e., the elements of the countable set
\begin{equation}\label{Omega*}
 \Omega_*=\left\{(q_1,q_2)\in \mathbb{Q}^2 \,:\,
   1/2<q_1<1/\sqrt{2},\, q_2>0,\,  q_1^2+q_2^2< 1/2\right\}.
     \end{equation}
Thus, for every $q=(q_1,q_2)\in \Omega_*$, there is a unique $(a,b)\in \Sigma$
such that $\Phi_1(a,b)=-q_1$ and $\Phi_2(a,b)=q_2$.
 For every $(q_1,q_2)\in \Omega_*$, let
\begin{equation}\label{5.1}
 \Theta_1(t)=\int_0^t\frac{\mu}{\mu^2-k(u)^2}du,\quad
  \Theta_2(t)=\int_0^t\frac{\upsilon}{\upsilon^2-k(u)^2}du
  \end{equation}
and
\begin{equation}\label{5.2}
 r(t)=\sqrt{\mu^2-\upsilon^2}k(t)+\upsilon\sqrt{\mu^2-k(t)^2}
 \cos\Theta_1(t),
 \end{equation}
where $a,b$ are the parameters of $q$ and $k$, $\mu$, $\upsilon$ stand for
$k_{a,b}$, $\mu(a,b)$ and $\upsilon(a,b)$, respectively. Note that $(q_1,q_2)$
and $(a,b)$ are related by
\begin{equation}\label{5.3}
 q_1=
 \frac{1}{2\pi}\int_0 ^{\omega} \Theta_1'(t) dt,\quad q_2=-\frac{1}{2\pi}\int_0 ^{\omega} \Theta_2'(t)dt,
  \end{equation}
where $\omega=\omega(a,b)$ is the minimal period of $k$.

\begin{defn}\label{def:symm-conf}
The {\it symmetrical configuration} of the conformal strings with modulus $q=(q_1,q_2)$
is the parametrized curve $\gamma_{q}=(x,y,z) : \R\to \R^3$, defined by
\begin{equation}\label{5.4}
\begin{cases}
x(t)=\frac{\sqrt{2}}{r(t)}\mu\sqrt{k(t)^2-\upsilon^2}\cos\Theta_2(t),\\
y(t)=\frac{\sqrt{2}}{r(t)}\mu\sqrt{k(t)^2-\upsilon^2}\sin\Theta_2(t),\\
z(t)=\frac{\sqrt{2}}{r(t)}\upsilon\sqrt{\mu^2-k(t)^2}\sin\Theta_1(t).\\
\end{cases}
\end{equation}
\end{defn}

Theorem \ref{thm:B} is now a direct consequence of the following.

\begin{thm}\label{5.I}
 Let $q=(q_1,q_2)$ be any element of $\Omega_*$ and $(a,b)\in \Sigma_*$ be the
  corresponding parameters. Any conformal string with parameters $(a,b)$ is
   conformally equivalent to the symmetrical configuration $\gamma_{q}$.
    \end{thm}

\begin{proof}
Consider the unique conformal parametrization $\gamma :\R\to \R^3$ of a conformal
string with parameters $a$, $b$ whose Vessiot frame $\mathrm{F}$ satisfies the
initial condition $\mathrm{F}(0)=\mathrm{Id}$. It suffices to prove that $\gamma$ is
conformally equivalent to $\gamma_q$.
By Lemma \ref{RL}, the canonical lift $\Gamma:\R\to \mathcal{L}_+$
of $\gamma$ takes the form $\tilde{Y} W$, where $\tilde{Y}\in GL(5,\C)$ and
$W={}^t\!(w_0, \dots, w_4)$ is the $\C^5$-valued map defined by
\begin{equation}\label{5.5.a}
\begin{aligned}
w_0&=k,\quad
 w_1=\sqrt{\mu^2-k^2}e^{i\Theta_1(t)},\quad
  w_2=\sqrt{\mu^2-k^2}e^{-i\Theta_1(t)},\\
   w_3&=\sqrt{\upsilon^2-k^2}e^{i\Theta_2(t)},\quad
    w_4=\sqrt{\upsilon^2-k^2}e^{-i\Theta_2(t)},
      \end{aligned}
     \end{equation}
On the other hand, the curve
$\widetilde{\Gamma}={}^t\!(\widetilde{\gamma}_0, \dots, \widetilde{\gamma}_4) :\R\to \mathcal{L}_+$,
defined by
\begin{equation}\label{5.5}
\begin{cases}
\widetilde{\gamma}_0(t)=\frac{1}{\sqrt{2}}\left(
\sqrt{\mu^2-\upsilon^2}k(t)-\upsilon\sqrt{\mu^2-k(t)^2}\cos\Theta_1(t)\right),\\
\widetilde{\gamma}_1(t)=\mu\sqrt{k(t)^2-\upsilon^2}\cos\Theta_2(t),\\
\widetilde{\gamma}_2(t)=\mu\sqrt{k(t)^2-\upsilon^2}\sin\Theta_2(t),\\
\widetilde{\gamma}_3(t)=\upsilon\sqrt{\mu^2-k(t)^2}\sin\Theta_1(t),\\
\widetilde{\gamma}_4(t)=\frac{1}{\sqrt{2}}\left(
\sqrt{\mu^2-\upsilon^2}k(t)+\upsilon\sqrt{\mu^2-k(t)^2}\cos\Theta_1(t)\right),
\end{cases}
\end{equation}
is a null lift of $\gamma_{q}$. From \eqref{5.5.a} and \eqref{5.5}, it follows
that $W=\mathrm{Z} \widetilde{\Gamma}$, for some $\mathrm{Z}\in GL(5,\C)$.
Consequently, $\Gamma=\mathrm{L} \widetilde{\Gamma}$, for a suitable
$\mathrm{L}\in GL(5,\C)$, which yields $\mathrm{F}=\mathrm{L} \mathrm{F}_{q}$,
where $\mathrm{F}_{q}$ is the Vessiot frame along $\gamma_{q}$. Thus $\mathrm{L}\in G_+$,
which implies that $\gamma$ and $\gamma_{q}$ are equivalent to each other.
\end{proof}

\begin{remark}
The map $\Phi$ can be inverted by numerical methods. Once we know the parameters $a$ and
$b$ which correspond to the modulus $q$, we can use~\eqref{5.4} to find the explicit
parametrization of the symmetrical configuration $\gamma_{q}$.
\end{remark}

\section{The proof of Theorem \ref{thm:C}}\label{s:5}

Let $\Omega_*$ be the countable set introduced in \eqref{Omega*}.
For every $q=(q_1,q_2)\in \Omega_*$, let $q_1=m_1/n_1$ and $q_2=m_2/n_2$,
with $(m_1,n_1)=(m_2,n_2)=1$. Let $n$ denote the least common multiple of $n_1$ and $n_2$.
Then, $n=h_1n_1$ and $n=h_2n_2$, where $h_1$ and $h_2$ are two positive coprime
integers. Let $\gamma_q$ be the symmetrical configuration corresponding to $q$ (cf. Definition
\ref{def:symm-conf}).
We are now ready to prove Theorem \ref{thm:C} asserting that (1) $n$ is the order of the
symmetry group of $\gamma_q$ and (2) $m_1h_1$ and $m_2h_2$ are, respectively, the linking numbers
of $\gamma_q$ with the Clifford circle $\mathcal{C}=\left\{(x,y,0) \in\R^3 : x^2+y^2=2\right\}$
and with the $z$-axis.

\begin{proof}[Proof of Theorem \ref{thm:C}]

The curve $\gamma_q$ is a real-analytic closed curve with positive chirality (cf. Section \ref{ss:1.2},
Definition \ref{def:monodromy-chirality}).
Therefore, the symmetry group of $\gamma_q$ is generated by the monodromy of $\gamma_q$. By construction,
the canonical null lift $\Gamma:\R\to \mathcal{L}_+$ is as in \eqref{5.5}
and the first conformal curvature is a strictly positive periodic function,
with minimal period $\omega$. Using~\eqref{5.5}, we find
$\Gamma_{q}(t+\omega)=R(2\pi q_2,2\pi q_1) \Gamma_{q}(t)$, where
$R(2\pi q_2,2\pi q_1)$ is as in \eqref{1.1.6}. Therefore, $R(2\pi q_2,2\pi q_1)$
is the monodromy of $\gamma_q$. This implies that the symmetry group of $\gamma_q$ is
generated by $R(2\pi q_2,2\pi q_1)$. Let $[z]$ denote the $z$-axis with the downward
orientation induced by the parametrization $\alpha(s)=(0,0,-s)$, and let $[\mathcal{C}]$
be the Clifford circle equipped with the orientation induced by the
rational parametrization
\[
 \beta : t\in \R \mapsto (-\sqrt{2}(t^2-2)/(t^2+2), 4t/(2+t^2), 0)\in \R^3.
  \]
Now we compute the Gauss linking integrals $\mathrm{lk}(\gamma_q,[z])$
and $\mathrm{lk}(\gamma_q,[\mathcal{C}])$.
If $\gamma_q(t) = (x(t),y(t),z(t))$,
the Gauss linking integral $\mathrm{lk}(\gamma_q,[z])$ is given by
\[
\begin{split}
\mathrm{lk}(\gamma_q,[z])&=\frac{1}{4\pi} \int_{[\gamma_q]}\int_{z}
\left\langle \frac{\gamma_q-\alpha}{\|\gamma_q-\alpha\|^3},  d\gamma_q\times d\alpha\right \rangle\\
&=-\frac{1}{4\pi} \int_{0}^{n\omega}\left(\int_{-\infty}^{+\infty}\left(\frac{x(t)y'(t)-x'(t)y(t)}{\left[x(t)^2
  +y(t)^2+(z(t)+ s)^2\right]^{3/2}}\right)ds\right)dt\\
&=-\frac{1}{2\pi} \int_0^{n\omega}\frac{x(t)y'(t)-x'(t)y(t)}{x(t)^2+y(t)^2}dt\\
 &= -\frac{n}{2\pi}\int_0^{\omega}\Theta_2'(t)dt = n\Phi_2= n q_2 = h_2m_2.
\end{split}
\]
To compute the linking integral of $\gamma_q$ with the Clifford circle we consider
the o\-rienta\-tion-\-pre\-serving conformal involution
\begin{align}
  \Psi:(x,y,z)\in \R^3&\mapsto
   \frac{1}{x^2+y^2+z^2+2\sqrt{2}x+2}\\
 \notag  &\quad\times\left(x^2+y^2+z^2+2\sqrt{2},4z,4y\right)\in \R^3.
   \end{align}
This map takes $(-\sqrt{2},0,0)$ to the point at infinity
and exchanges the roles of the two axes of symmetry, i.e., $\Phi\circ \beta=-\alpha$.
A direct computation shows that the parametric equations of $\gamma_* :=\Phi\circ \gamma_q$
are
\[
\begin{cases}
x_*(t)=\frac{\sqrt{2}}{r^*(t)}\upsilon\sqrt{\mu^2-k(t)^2}\cos \Theta_1(t),\\
y_*(t)=\frac{\sqrt{2}}{r^*(t)}\upsilon\sqrt{\mu^2-k(t)^2}\sin \Theta_1(t),\\
z_*(t)=\frac{\sqrt{2}}{r^*(t)}\mu\sqrt{k(t)^2-\upsilon^2}\sin \Theta_2(t),
\end{cases}
\]
where $r^*(t)=\sqrt{\mu^2-\upsilon^2}k(t)+\mu\sqrt{k(t)^2-\upsilon}\cos \Theta_2(t) $.
We then have
\[
\begin{split}
\mathrm{lk}(\gamma_q,[\mathcal{C}])&=-\mathrm{lk}(\gamma_*,[z])=
-\frac{1}{4\pi} \int_{[\gamma_*]}\int_{z}
\left\langle \frac{\gamma_*-\alpha}{\|\gamma_*-\alpha\|^3},  d\gamma_*\times d\alpha\right\rangle\\
&=\frac{1}{4\pi} \int_{0}^{n\omega}\left(\int_{-\infty}^{+\infty}
\left(\frac{x_*(t)y_*'(t)-x_*'(t)y_*(t)}{\left[x_*(t)^2+y_*(t)^2+
(z_*(t)+ s)^2\right]^{3/2}}\right)ds\right)dt\\
&=\frac{1}{2\pi} \int_0^{n\omega}\frac{x_*(t)y_* '(t)- x_* '(t)y_*(t)}{x_*(t)^2+y_*(t)^2}dt\\
 &= \frac{n}{2\pi}\int_0^{\omega}\Theta_1'(t)dt = -n\Phi_1= nq_1 = h_1m_1,
\end{split}
\]
which yields the required result.
\end{proof}

\begin{figure}[ht]
\begin{center}
\includegraphics[height=6cm,width=6cm]{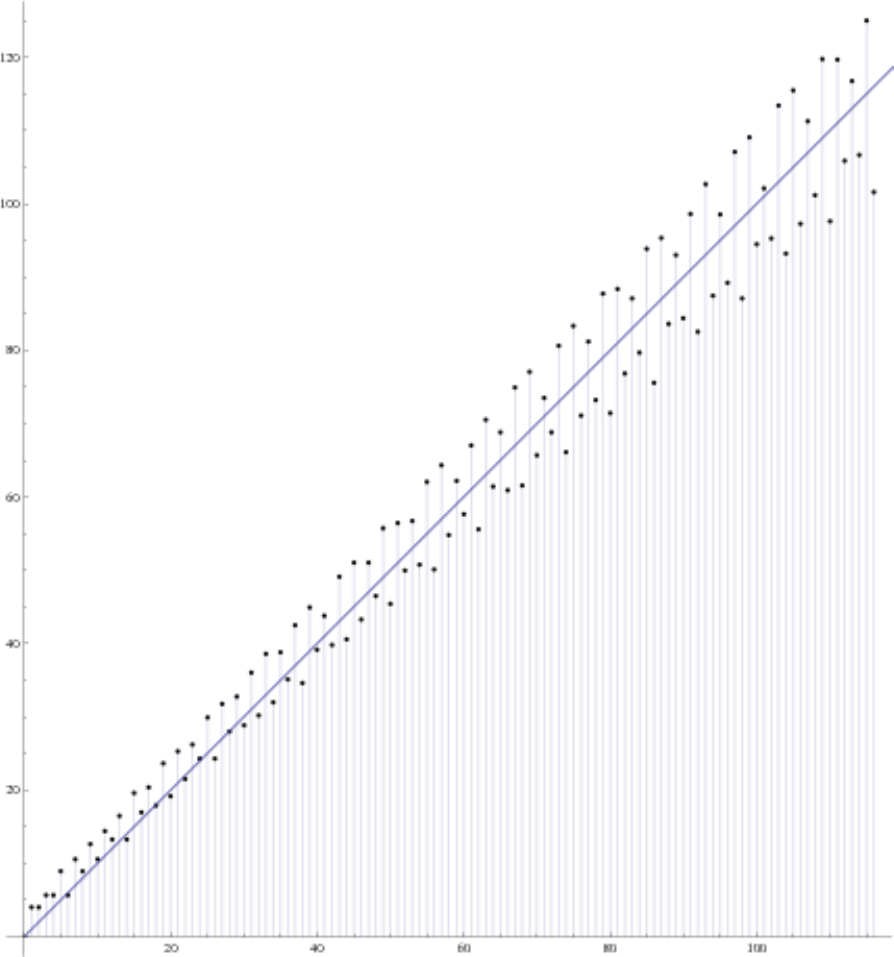}
\caption{The graph of the function $4\sqrt{\varrho(n)}$, $5\le n \le 120$.}\label{FIG3}
\end{center}
\end{figure}

\section{Final remarks and examples}\label{s:6}

In this section we discuss some examples
and comment on some numerical experiments carried out
with the software {\sl Mathematica}.

\subsection{Euclidean and Clifford symmetries}
Let $\gamma$ be the symmetrical configuration with modulus $q=(q_1,q_2)$.
Then, $E(q_1,q_2)=R(2\pi q_2,2\pi q_1)^{n_1}$ is the Euclidean rotation of
angle $2\pi m_2 n_1/n_2$ around the $z$-axis. The subgroup of order $h_1$ generated
by $E(q_1,q_2)$ is the {\it Euclidean symmetry group} of the symmetrical configuration.
Similarly, $C(q_1,q_2)=R(2\pi q_2,2\pi q_1)^{n_2}$ is the toroidal rotation of angle
$2\pi m_1 n_2/n_1$ around the Clifford circle. The cyclic subgroup of order $h_2$
generated by $C(q_1,q_2)$ is the {\it Clifford symmetry group} of the symmetrical
configuration. Denoting by $\omega$ the conformal wavelength of $\gamma$, the arc
$[\gamma]_0=\gamma([0,\omega))\subset [\gamma]$ is the {\it fundamental domain} of
the string and the trajectory $[\gamma]$ is obtained from $[\gamma]_{0}$
via $G_{\gamma}$, i.e.,
$
 [\gamma]=\bigcup_{A\in G_{\gamma}} A\cdot [\gamma]_0
  $.

\subsection{Quantum numbers} The symmetrical configurations are labeled by three quantum numbers:
the order of the symmetry group and the linking numbers with the Clifford circle and
the $z$-axis. We use the notation $|n,\ell_1,\ell_2>$ for the symmetrical
configuration with a symmetry group of order $n$ and linking numbers
$\mathrm{lk}(\gamma,[\mathcal{C}])=\ell_1$ and $\mathrm{lk}(\gamma,[z])=\ell_2$, respectively.
Let $\varrho(n)$ be the cardinality of the set of the symmetrical configurations
with index of symmetry $n$. There are no symmetrical configurations
with $n=1,2,3,4$. Figure \ref{FIG3} reproduces the graph of the function
$4\sqrt{\varrho}$, for $5\leq n \leq 120$.
This suggest that $\varrho$ has an asymptotic quadratic growth.

\begin{figure}[h]
\begin{center}
\includegraphics[height=6cm,width=6cm]{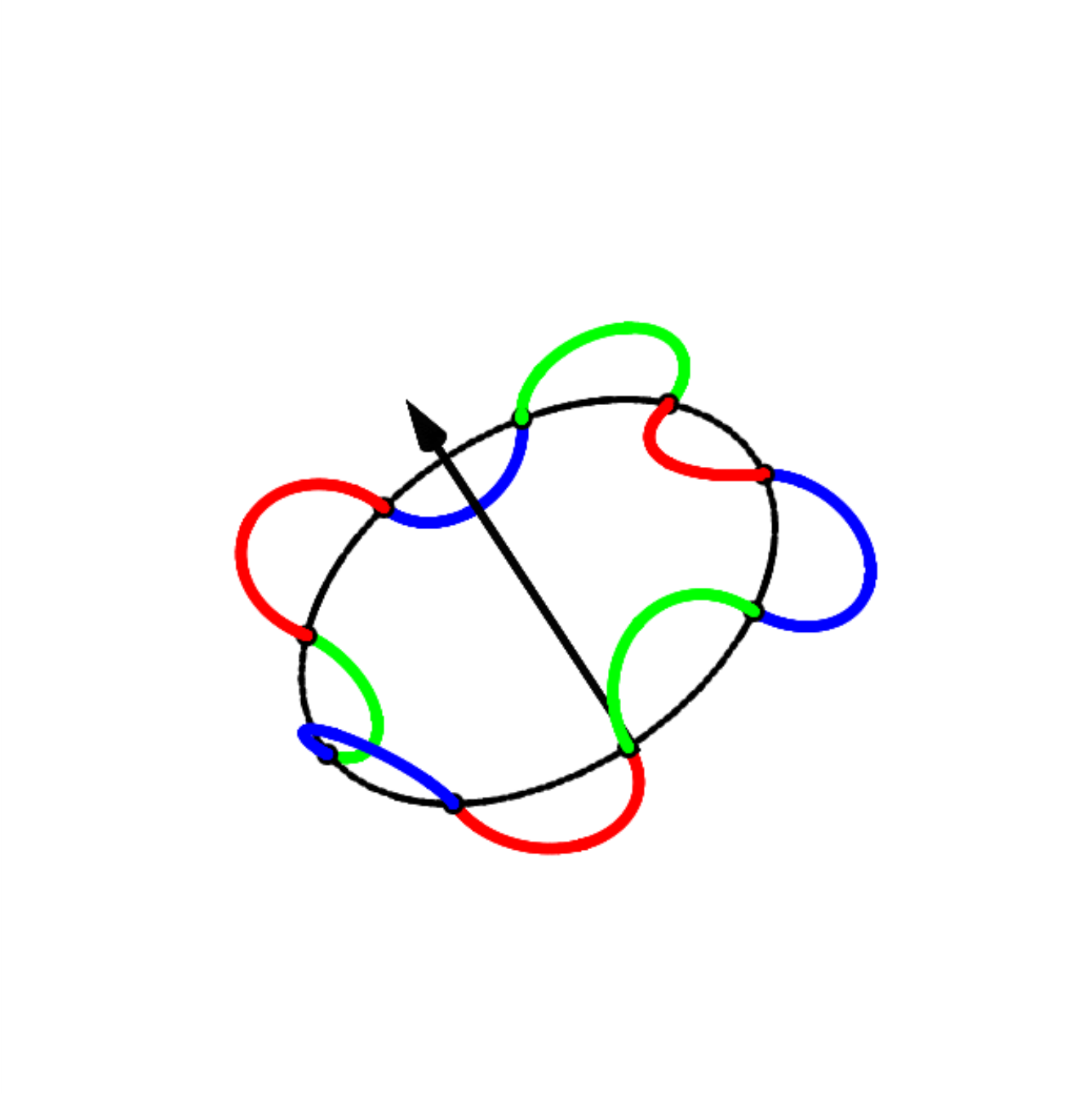}
\includegraphics[height=6cm,width=6cm]{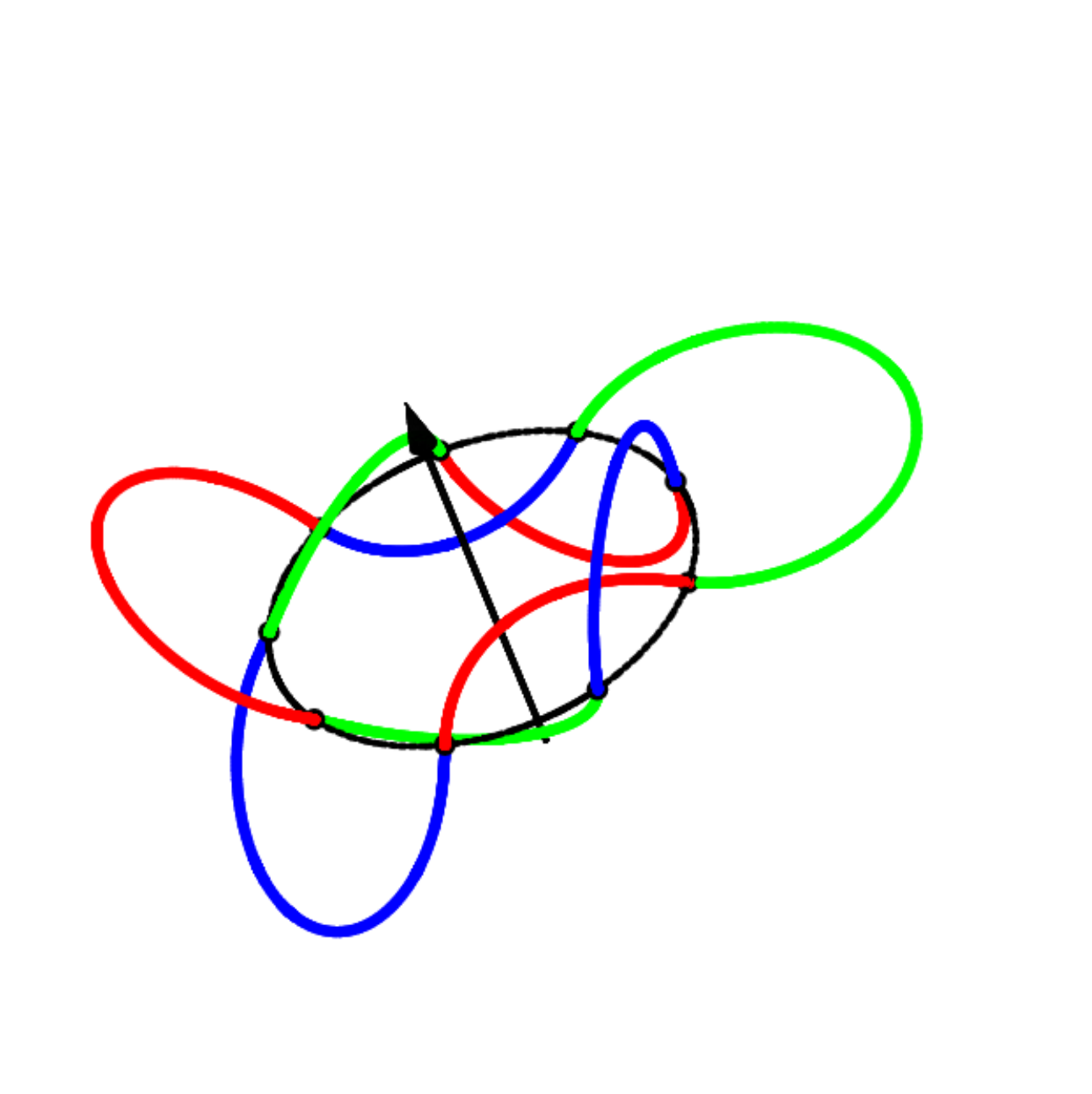}\\
\vspace{-5em}
\includegraphics[height=6cm,width=6cm]{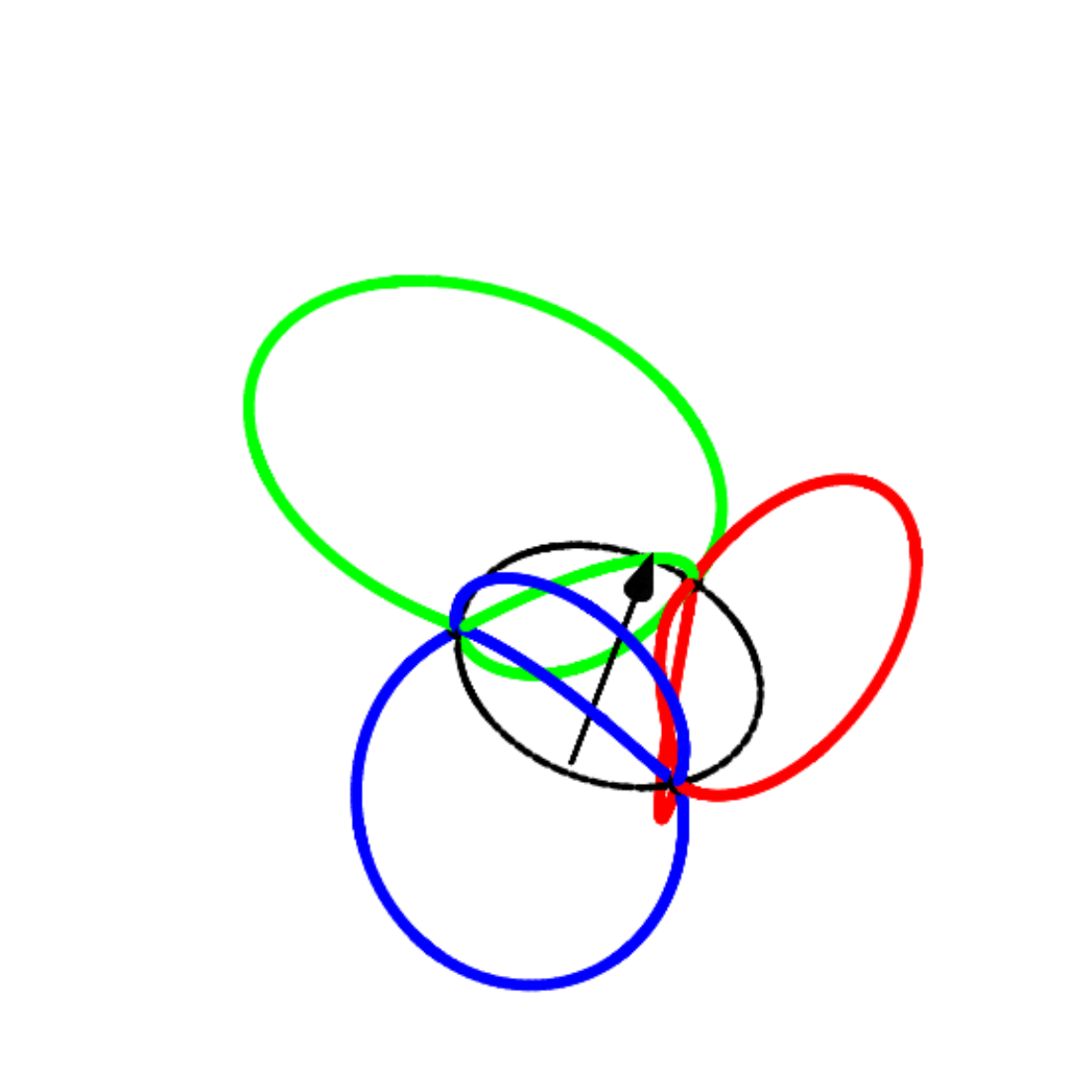}\vspace{-1em}
\caption{Symmetrical configurations $|9,5,1>$, $|9,5,2>$, and $|9,5,3>$, respectively.}\label{FIG4}
\end{center}
\end{figure}

\begin{figure}[h]
\begin{center}
\includegraphics[height=6cm,width=6cm]{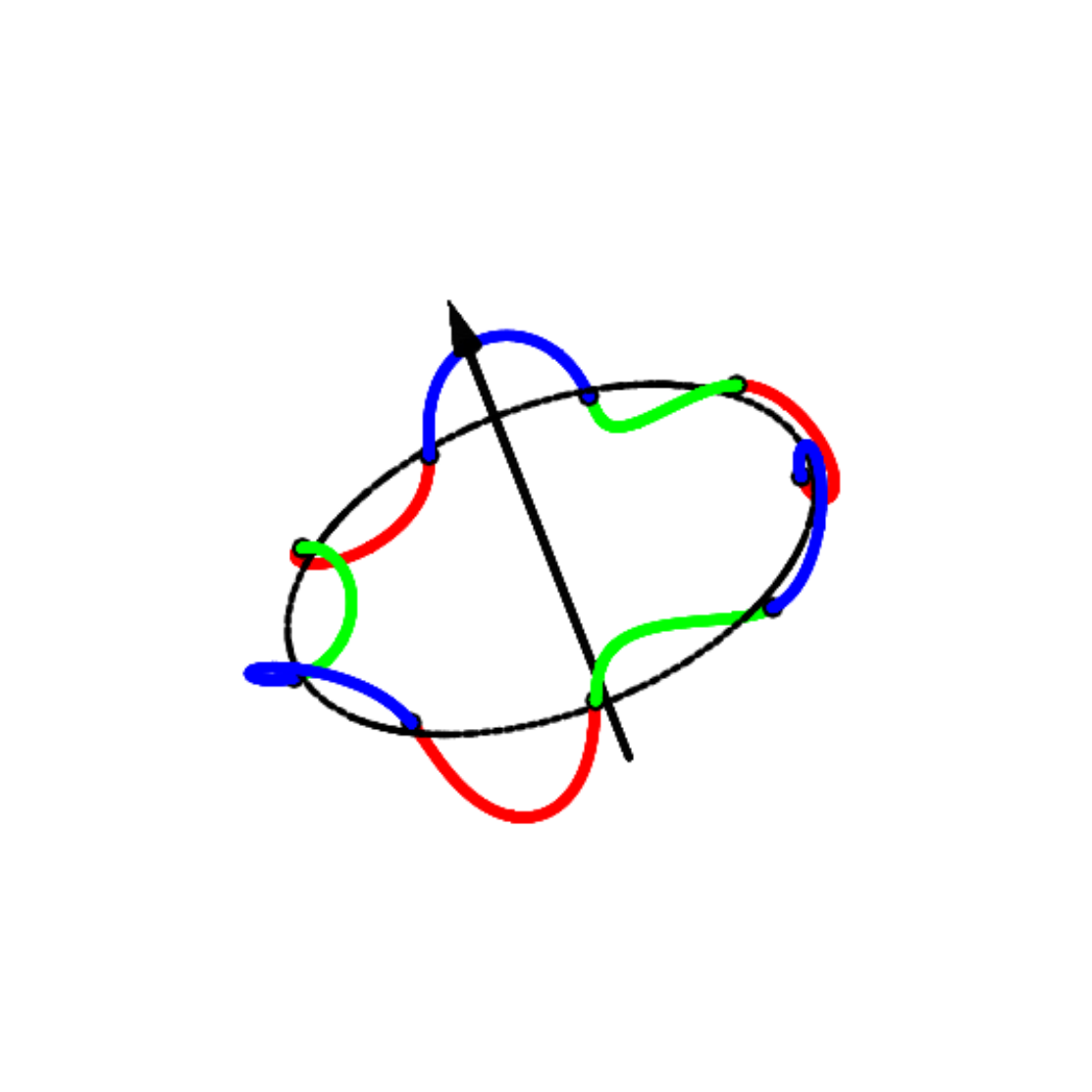}
\includegraphics[height=6cm,width=6cm]{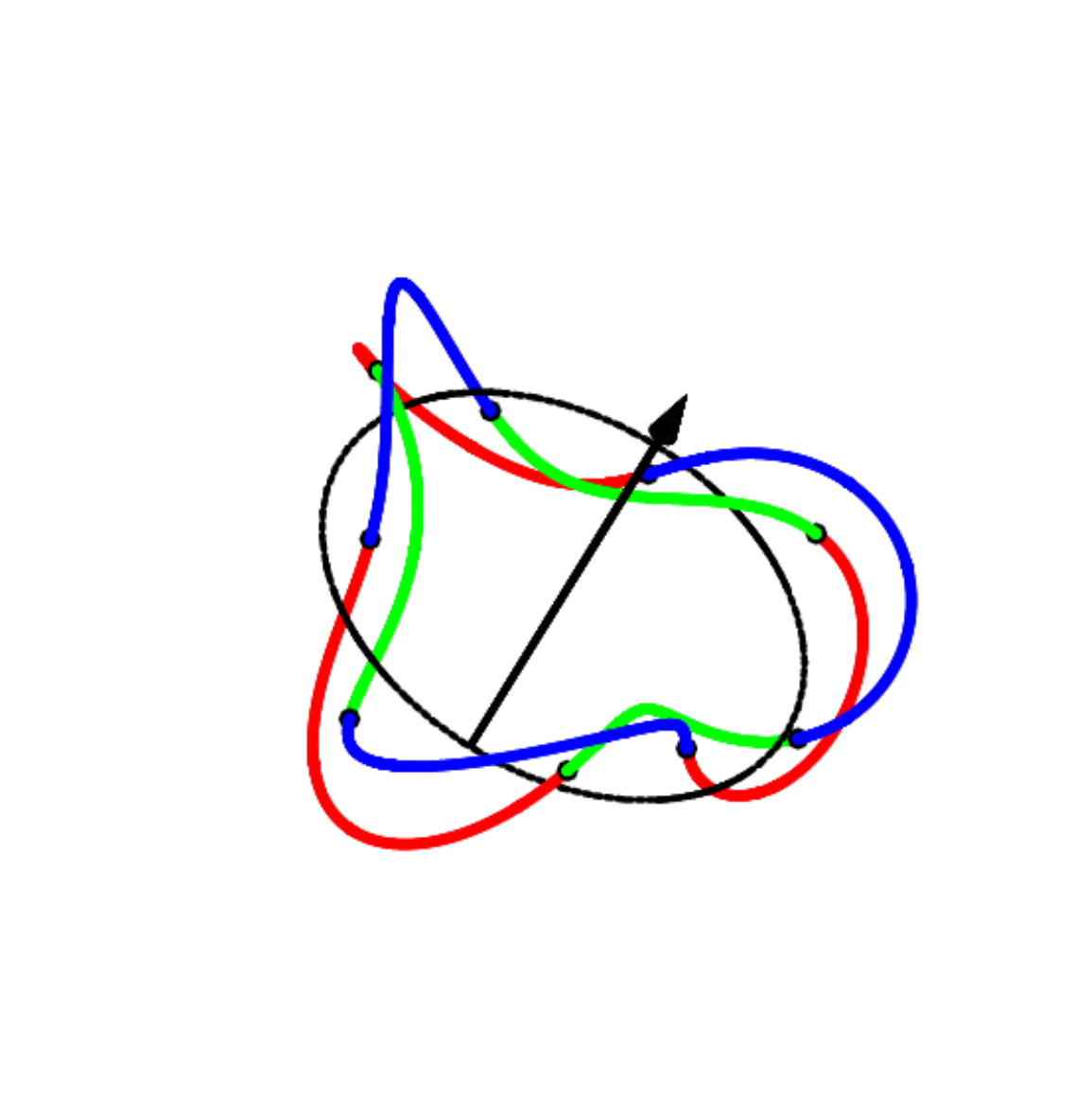}\vspace{-2em}
\caption{The symmetrical configurations $|9,6,1>$ and $|9,6,2>$, respectively.}\label{FIG5}
\end{center}
\end{figure}

\begin{figure}[h]
\begin{center}
\includegraphics[height=6cm,width=6cm]{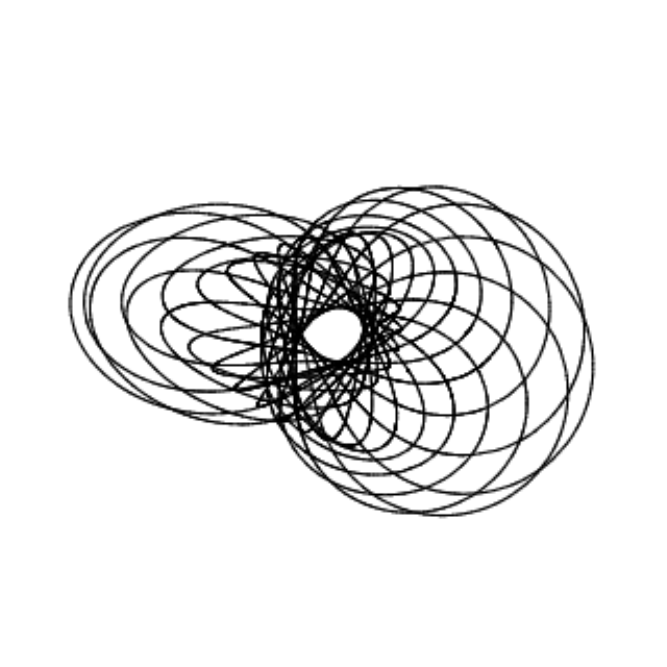}\vspace{-3em}
\caption{The symmetrical configuration $|63,37,24>$.}\label{FIG6}
\end{center}
\end{figure}

\begin{ex}
In this final example, we consider the symmetrical configurations with $n=9$.
This set consists of five elements: $|9,5,1>$, $|9,5,2>$, $|9,5,3>$, $|9,6,1>$,
and $|9,6,2>$.
The moduli of these strings are $(5/9,1/9)$, $(5/9,2/9)$, $(5/9,1/3)$, $(2/3,1/9)$,
and $(2/3,2/9)$, respectively. The Euclidean and the Clifford symmetry subgroups of the first
two strings are trivial. The Euclidean symmetry group of the third string is trivial, while
the Clifford symmetry group has order 3 (see Figure \ref{FIG4}). The fourth and fifth
strings have an Euclidean symmetry group of order 3 and have no non-trivial Clifford
symmetries (see Figure \ref{FIG5}). The invariants $a$, $b$ and the conformal wavelength
$\omega$ can be tabulated against the quantum numbers $n$, $\ell_1$ and $\ell_2$.
For instance, the values of the natural parameters $a$, $b$ and those of the conformal
wavelength $\omega$ of the stings with $n=9$ are
\[
\begin{aligned}
<a|9,5,1>&\approx 18.6403, \quad <b|9,5,1>\approx 0.06069,\quad <\omega|9,5,1> \approx 1.96996,\\
<a|9,5,2>&\approx 16.9699 \quad <b|9,5,2>\approx 0.09982,\quad <\omega|9,5,2> \approx  1.92188,\\
<a|9,5,3>&\approx 13.3269, \quad <b|9,5,3>\approx 0.29125 ,\quad  <\omega|9,5,3> \approx  1.81376,\\
<a|9,6,1>&\approx 2.6203, \quad <b|9,6,1>\approx 0.42577,\quad  <\omega|9,6,1>\approx  2.90618,\\
<a|9,6,2>&\approx 1.7209, \quad <b|9,6,2>\approx 0.90777,\quad  <\omega|9,6,2>\approx  2.79219.
\end{aligned}
\]

The shape of the strings becomes more complicated when $n$, $\ell_1$ and $\ell_2$ increase, see
for instance Figure \ref{FIG6}, where the symmetrical configuration with $n=63$,
$\ell_1=37$ and $\ell_2=24$ is reproduced. This string has no non-trivial Euclidean symmetries and
the subgroup of its Clifford symmetries has order~$3$.
\end{ex}

\end{document}